%% file: main.tex
\documentclass[oneside,a4paper]{amsart}

\input{preamble}
\input{macros}

\newtheorem{theo}{Theorem}
\theoremstyle{plain}

\newtheorem*{thm*}{Theorem}
\newtheorem*{lem*}{Lemma}
\newtheorem*{prop*}{Proposition}
\newtheorem*{cor*}{Corollary}

\theoremstyle{definition}

\newtheorem{cons}[theo]{Construction}

\newtheorem{quest}[theo]{Question}

\begin{document}

\title{A note on weight filtrations at the characteristic 
}
\author{Toni Annala}
\author{Piotr Pstr\k{a}gowski}

\begin{abstract}
We show that $\kgl$-linear cohomology theories over an affine Dedekind scheme $S$ admit a canonical weight filtration on resolvable motives without inverting residual characteristics. Combined with upcoming work of Annala--Hoyois--Iwasa, this endows essentially all known logarithmic cohomology theories with weight filtrations when evaluated on projective sncd pairs $(X,D)$ over $S$. Furthermore, the weight-filtered cohomology is an invariant of the open part $U = X-D$. 

On variants of de Rham cohomology, we show that our weight filtration recovers the décalaged pole-order filtration defined by Deligne. One interpretation of this is that the spectral sequence associated to the pole-order filtration is an invariant of $U$ from the $E_2$-page onwards, which generalizes a result of Deligne from characteristic 0 to positive and mixed characteristic, and suggests that ``mixed Hodge theory'' is a useful invariant of $S$-schemes. 

Finally, we compute explicit examples of weight filtered pieces of cohomology theories. One of the computations reproves a slight weakening of a result of Thuillier stating that the singular cohomology of the dual complex associated to the boundary divisor of a good projective compactification does not depend on the chosen  compactification.


In the appendix, we prove the folklore results that the Whitehead tower functor is fully faithful and that perfect bivariant pairings with respect to the twisted arrow category correspond to duality. 
\end{abstract}


\maketitle

\tableofcontents

\section{Introduction} 

The purpose of this article is to make the observation that various $\kgl$-linear cohomology theories carry a canonical weight filtration when restricted to what we call resolvable motives without inverting any of the residual characteristics. This is especially interesting when the so called logarithmic motive associated to a relative snc divisor in a smooth projective $\mathbf{Z}_{p}$-scheme (\textit{projective sncd pair over $\Zbf_p$}) is evaluated on the various cohomology theories coming from $p$-adic Hodge theory, such as algebraic de Rham cohomology, prismatic or syntomic cohomology, or localizing invariants such as topological cyclic homology. Interpreting the cohomology of a logarithmic motive as the corresponding logarithmic cohomology, our methods produce weight filtrations on logarithmic cohomology of projective relative sncd pairs in great generality. The motivic point of view also reveals that the weight-filtered logarithmic cohomology of $(X,D)$ is an invariant of the \emph{open part} $U = X-D$.




The construction is closely related to the work of Bondarko, who defined a weight structure on the category of Voevodsky's motives over a field and showed that this implies the existence of a canonical filtration on the values of any homological functor \cite{bondarko2010weight}, later refined to a functor valued in an appropriate filtered derived $\infty$-category in \cite{haine-pstragowski}. Here, we generalize this approach to the case of $\kgl$-linear cohomology theories defined on smooth schemes over Dedekind domains, where $\kgl$ is the effective $K$-theory spectrum. Many cohomology theories that occur in practice can be represented in this context, in particular:
\begin{enumerate}
    \item cohomology theories linear over motivic cohomology, such as motivic cohomology itself, syntomic cohomology, de Rham cohomology, Hodge cohomology, prismatic cohomology;
    \item localizing invariants, such as $K$-theory, $\THH$, $\TC$;
    \item filtered variants of the above, such as de Rham cohomology equipped with its Hodge filtration, and $\TC$ equipped by its motivic filtration. 
\end{enumerate}
An example of an important cohomology theory that is not $\kgl$-linear, and to which our methods therefore do not apply to, is given by algebraic cobordism $\mathrm{MGL}$.


\subsection{Statement of results} 

We write $S$ for an affine Dedekind scheme. We denote by $\DM^{\kgl}_S \colonequals \Mod_\kgl(\MS_S)$ the category $\kgl$-modules in motivic spectra of  \cite{AHI} and call it the category of \emph{$\kgl$-motives}. Its $\Abf^1$-invariant counterpart $\DM^{\kgl,\Abf^1}_S$ is defined analogously, but using Voevodsky's stable $\mathbf{A}^{1}$-homotopy category $\SH_{S} \subseteq \MS_{S}$ instead.  We say that a motive is
\begin{enumerate}
    \item \emph{perfect pure} (cf. \cite[Definition 3.1.1]{haine-pstragowski}) if it belongs to the smallest subcategory
    \[
    \Pure^\kgl_S \subseteq \DM^\kgl_S
    \]
    generated under extensions and retracts by Tate twists $\Mbf_{\kgl}(X)(n)$ of motives of smooth projective $S$-schemes $X$, where $n \in \mathbf{Z}$,\footnote{Note that since $\kgl$ is orientable, it is not necessary to allow general Thom spectra in the definition of a perfect pure motive, as they are all isomorphic to Tate twists.} 
    \item \emph{resolvable} if it belongs to the subcategory 
    \[
    \DM^\kgl_{S,\res} \subseteq \DM^{\kgl}_{S}
    \]
    generated by perfect pure motives under finite limits, colimits and retracts.
\end{enumerate}

Resolvable motives are dualizable by Atiyah duality \cite[Corollary~5.15]{AHI:atiyah}, and  therefore $\Abf^1$-invariant \cite[Proof of Proposition~6.21]{AHI:atiyah}. It follows that we have a fully faithful inclusion 
\begin{equation}
    \DM_{S,\res}^\kgl \subset \DM_{S}^{\kgl, \Abf^1}.
\end{equation}
Thus, we may regard resolvable motives as objects of either $\DM_{S}^{\kgl}$ or $\DM_{S}^{\kgl, \Ab^1}$, whichever is more convenient for the situation at hand.

The main technical observation of our paper is the following result. Throughout the article, $\Fil_\up(\Cc)$ stands for the category of objects of $\Cc$ equipped with an increasing filtration.

\begin{theorem}
\label{thm:IndependenceOfWts}
Let $S$ be an affine Dedekind scheme, $\Cc$ be a stable $\infty$-category with a $t$-structure $\tau$, and let $E \colon \DM^\kgl_S \to \Cc$ be an exact functor. Then it admits a unique enhancement to an exact functor taking values in filtered objects
\[
W_{\ast} E \colon \DM^\kgl_{S,\res} \to \Fil_\up(\Cc),
\]
called the \emph{$\tau$-weight filtration}, such that 
\[
(W_{\ast} E)(M) \simeq \tau_{\geq -\ast}(E(M))
\]
for all perfect pure $M$.
\end{theorem}

We have introduced the minus sign above to be consistent with the convention in literature that weight filtrations are increasing. In the main body of the paper, we show a slightly stronger result, namely that any additive functor valued in a stable $\infty$-category defined on the subcategory of perfect pure motives extends uniquely to an exact functor on all resolvable motives, see \cref{prop:PureToRes}. 

In practice, \cref{thm:IndependenceOfWts} will be applied to functors of the form $E(-) \colon \DM^{\kgl}_{S,\res} \to \spectra$, or variants such as functors valued in filtered spectra,  represented by a (usually non-resolvable) $\kgl$-linear motivic spectrum $E$. The difference to prior art (see e.g. \cite{bondarko2010weight},  \cite{haine-pstragowski}) is that we allow $S$ to be a Dedekind scheme whose residual characteristic need not be invertible in the coefficients of $E$. Examples of cohomology theories that can be represented by a $\kgl$-linear motivic spectrum and that contain interesting at-the-characteristic information include $H_\mot, H_\dR, H_\crys, H_\prism, K, \THH$, and $\TC$, as well their filtered variants.\footnote{For the $p$-complete examples in this list, the $\kgl$-linear representability follows from \cite[Corollary~6.13]{AHI:atiyah}. The argument for the non-$p$-complete examples (e.g. de Rham cohomology, $\TC$) is similar, but uses the commutative algebra maps of filtered sheaves spectra constructed by Bouis \cite{bouis:2024}. We provide a detailed argument for the representability of de Rham cohomology in \S\ref{subsection:de_rham_and_hodge_cohomology_as_motivic_spectra}.} 

In order to understand the geometric significance of \cref{thm:IndependenceOfWts}, we recall that if $(X,D)$ is a projective strict normal crossing divisor pair (\textit{sncd pair}) of dimension $n$ over $S$,\footnote{Recall that $D = D_1+\cdots+D_r \subseteq X$ is a relative snc divisor if it is snc divisor and all of its components $D_i$ are smooth over $S$.} then its \emph{logarithmic motive}, defined as the total fibre
\begin{equation}
    \Mbf_{\kgl}(X,D) \colonequals \tfib_{I \subset [r]} (\Mbf_{\kgl}(D_I)(-n)^\vee) \in \DM^\kgl_{S,\res}
\end{equation}
of the hypercube given by the dual of the Tate-twisted canonical diagram of $D_I = \bigcap_{i \in I} D_i$, is resolvable.\footnote{The one-morphisms of this hypercube coincide with Tang's Gysin maps \cite{tang:gysin}. Thus, the logarithmic motive is assembled from ordinary motives using Gysin maps. This is similar to logarithmic homotopy theory where Gysin sequences can be used to alter the logarithmic structure of a motive \cite[Theorem~1.3.5]{binda:2023}.} As the terminology suggests, evaluating a motivic spectrum $E$ on a logarithmic motive corresponds to computing the logarithmic $E$-cohomology of the sncd pair $(X,D)$.\footnote{It was proven in \cite[Proposition~6.27]{AHI:atiyah} that in the cases where $E$ represents the crystalline and the de Rham cohomology, $E(\Mbf_{\kgl}(X,D))$ coincides with the logarithmic crystalline and logarithmic de Rham cohomology of the pair $(X,D)$ as they are traditionally defined. A similar comparison result holds in general (e.g. for $H_\prism, \THH$, and $\TC$ considered in \cite{rognes2009topological, koshikawa2020logarithmic, binda2023logarithmic}), and its proof is the subject to upcoming joint work of the first author with Marc Hoyois.} Thus, \cref{thm:IndependenceOfWts} may be regarded as a machine for producing weight filtrations for logarithmic cohomology theories. This is already quite interesting in the case of positive characteristic, where for example it equips logarithmic crystalline cohomology with a functorial weight filtration independently of the results of \cite{nakkajima2008weight}. For many other cohomology theories, such as logarithmic prismatic cohomology \cite{koshikawa2020logarithmic,binda2023logarithmic}, no weight filtration has been constructed previously.


The motivic point of view not only provides a convenient way to define weight filtrations on logarithmic cohomology, but importantly also shows that they are invariants of the \emph{open part} $U = X - D$. Indeed, essentially due to the Morel--Voevodsky purity theorem \cite{morel:1999}, the logarithmic motive admits a simple geometric description in the presence of $\Abf^1$-invariance, namely
\begin{equation}
    \Mbf_{\kgl}(X,D) \simeq \Mbf^{\Abf^1}_{\kgl}(U) \in \DM^\kgl_{S, \res}
\end{equation}
(see \cite[Lemma~6.24]{AHI:atiyah}), where $\Mbf^{\Abf^1}_{\kgl}(U)$ is the $\Abf^1$-invariant $\kgl$-motive of $U$. In particular,  $\Mbf_{\kgl}(X,D)$ is an invariant of $U$. As a consequence of this, we may strengthen \cref{thm:IndependenceOfWts} to the following result.


\begin{theorem}
Let $S$ be an affine Dedekind scheme, and let $(X,D)$ be a projective sncd pair over $S$. Then the filtered object
\begin{equation}
    (W_*E)(\Mbf_{\kgl}(X,D)) \in \Fil_\up(\Cc)
\end{equation}
defined in \cref{thm:IndependenceOfWts} is an invariant of $U = X-D$.
\end{theorem}

Next, we compare our weight filtration with the filtration by pole orders on the logarithmic de Rham cohomology \cite{deligne:HodgeII}. Unfortunately, our comparison is proven to be functorial only in those maps of sncd pairs that do not change the ambient variety. This is because our proof strategy is ill-suited for establishing functoriality. We outline an alternative strategy that should give a more functorial comparison result in \cref{rem:OurFailures}. We denote by $\dR(S) \in \DM^\kgl_S$ the $\kgl$-motive that represents algebraic de Rham cohomology over $S$ (see \cref{definition:hodge_filtered_de_rham_cohomology_and_its_cousins}). 

\begin{theorem}[\cref{thm:DeligneCompare}]
Let $S$ be an affine Dedekind scheme, and let $(X,D)$ be a projective sncd pair over $S$. Then there exist an equivalence 
\begin{align}
D_* \R\Gamma(X; \Omega^\bullet_{X/S}(\log D)) &\simeq W_*\dR(S)(\Mbf_{\kgl}(X,D)) \in \Fil_\up(\D(S)), 
\end{align}
where the left side is the décalaged pole-order filtration on logarithmic de Rham cohomology. In particular, the left side is an invariant of the open part $U = X - D$.
\end{theorem}

We also prove similar result for the Hodge filtered de Rham cohomology and for the Hodge cohomology. This result can be regarded as a generalization of Deligne's theorem from characteristic 0 that states that, for a projective sncd pair $(X,D)$, the spectral sequence induced by the pole-order filtration is an invariant of $U$ from the $E_2$-page onward. It suggests that ``mixed Hodge theory'' is an interesting invariant of non-projective schemes in positive and mixed characteristic, not only in characteristic 0.

We end this article in \cref{sect:misc}, where we compute explicit examples of weight graded pieces of cohomology theories, explain how to extend weight filtration to singular schemes, and conjecture an alternative construction of the category of resolvable motives that would be useful for defining functors out of it. One of the computations (\cref{ex:bdry}) reproves a slight weakening of a result of Thuillier stating that the singular cohomology of the dual complex associated to the boundary divisor of a good projective compactification does not depend on the chosen  compactification.

\begin{remark}
Unlike in the case of $\ell$-adic cohomology theories studied in \cite{haine-pstragowski}, where the weight filtrations are constructed for an appropriate class of cohomology theories representable in $\SH(k)$, beware that \cref{thm:IndependenceOfWts} only applies to $\kgl$-linear cohomology theories is and the resulting weight filtration is only natural in $\kgl$-linear maps between such. This in particular means that a priori only $\kgl$-linear cohomology operations can be lifted to the weight-filtered context. A non-trivial example of such a cohomology operation at the characteristic is the first motivic Milnor operator $Q_1$ on mod-$p$ motivic and syntomic cohomology of smooth $\Fb_p$-varieties \cite{annala-elmanto, carmeli-feng}.

The key difference in the $p$-adic case is that while the needed connectivity result, \cref{lem:kglChow}, holds for $\kgl$-motives, we do not know if this vanishing can be extended to the case of $\MGL$-motives, or further to the case of the motivic sphere 
(in the modified form of \cite[Lemma 3.3.4]{haine-pstragowski}). Away from the characteristic these depend on in an essential way on the Hopkins-Morel-Hoyois equivalence of \cite{hoyois:2013}, which is not currently known to hold at the characteristic. 
\end{remark}

\subsection{Notation and conventions} 

In this paper, we will have to make use of both increasing (such as the negatively indexed Whitehead filtration $\tau_{\geq - \ast} X$ of a spectrum) and decreasing (such as the Hodge filtration $\Omega^{\geq \ast}$ on the de Rham complex) filtrations. To keep these apart, if $\ccat$ is an $\infty$-category, we write $\Fil^{\down}(\ccat) \colonequals \Fun(\Zb^{\op}, \ccat)$
for the $\infty$-category of decreasingly filtered objects and $\Fil_{\up}(\ccat) \colonequals \Fun(\Zb, \ccat)$ for the $\infty$-category of increasingly filtered objects. 

If $S$ is a derived scheme, we write $\D(S) \simeq \QCoh(S)$ for its derived $\infty$-category, $\DF(S) \colonequals \Fil^{\down}(\D(S))$ for the  filtered derived $\infty$-category, and $\DGr(S) \colonequals \Fun(\mathbb{Z}, \D(S))$ for the graded derived $\infty$-category. Note that according to our convention, the filtered derived $\infty$-category is \emph{decreasingly} filtered. 

\subsection{Acknowledgments}
We would like to thank Brian Shin for pointing out a mistake in the original proof of vanishing of $\kgl$ above the Chow line, Tess Bouis for explaining facts about motivic filtrations, as well as Ryomei Iwasa and Yuchen Wu for useful discussions.

\section{Construction of the weight filtration}

\begin{lemma}[Effective $K$-theory vanishes above the Chow line]
\label{lem:kglChow}
If $X \in \Sm_S$, then $\kgl^{i,j}(X) \cong 0$ for all $i > 2j$.
\end{lemma}
\begin{proof}
Bachmann proves that $\kgl := f_0 \KGL$ is a framed suspension spectrum, hence very effective \cite{bachmann:2024}.\footnote{Any framed motivic space is colimit of framed motivic spaces of form $L_\mathrm{mot} \gamma(X) \in \mathrm{H}^\mathrm{fr}_S$, where $X \in \Sm_S$~\cite[Proposition~3.2.10]{EHKSY1}. Moreover, the framed suspension spectrum of $L_\mathrm{mot} \gamma(X) \in \mathrm{H}^\mathrm{fr}_S$ coincides with the usual suspension spectrum of $L_\mathrm{mot} X \in \mathrm{H}_S$ in $\SH^\mathrm{fr}_S = \SH_S$, where the last identification is the reconstruction theorem \cite[Theorem~18]{hoyois-loc}. Thus, the very effectivity of framed suspension spectra follows from the fact that very effective motivic spectra are closed under colimits.} Thus, $\kgl$ is connective in the homotopy t-structure, and consequently for all $X \in \Sm_S$
\[
\kgl^{i,j}(X) \cong 0
\]
if $i > j+\dim(X)$. Moreover,
\begin{equation}\label{eq:kgltoHZ}
\kgl/\beta \simeq H\mathbf{Z},   
\end{equation}
where $\beta \in \kgl^{-2,-1}(S)$ is the Bott element \cite{bachmann:2024}. As motivic cohomology vanishes above the Chow line \cite{spitzweck:2018}, the long exact sequence arising from (\ref{eq:kgltoHZ}) implies that if $r>0$ then multiplication by powers of $\beta$ produce surjections $\kgl^{2j' + r, j'}(X) \tto \kgl^{2j + r, j}(X)$ where $j' \geq j \geq 0$. As $\kgl^{2j + r, j}(X) \cong 0$ if $j\geq\dim(X)$, the desired vanishing follows.
\end{proof}

\begin{lemma}
\label{lemma:connectivity_of_mapping_spectra_between_kgl_motives}
Let $A, C \in \Pure^\kgl_S$. Then the mapping spectrum
\[
\Hom_{\DM^\kgl_S}(A, C)
\] 
is connective. Moreover, every cofibre sequence
\[
A \to B \to C
\]
in $\DM^\kgl_S$ splits. 
\end{lemma} 

\begin{proof}
If $A, C$ are Tate twists of motives of a smooth projective variety, this is essentially the same as \cite[Lemma 2.2.13]{elmanto-sosnilo} but using \cref{lem:kglChow} instead of the analogous vanishing result for motivic cohomology. The general case follows. The second claim follows from the first by observing that every such cofibre sequence is classified by a map $C \to A[1]$ in $\DM^\kgl_S$ which is then necessarily nullhomotopic. 
\end{proof}

Thus, $\Pure^\kgl_S$ may be defined in an alternative way as the smallest retract-closed additive subcategory generated by the Tate twists of $\kgl$-motives of smooth projective $S$-schemes. As $\DM^\kgl_S$ is idempotent complete, so is $\Pure^\kgl_S$. Thus, $\Pure^\kgl_S$ is the heart of a Bondarko's \textit{weight structure} (see e.g. \cite{bondarko2010weight} or \cite[Definition~2.2.1]{elmanto-sosnilo}) as we explain next.

\begin{definition}
\label{definition:chow_weight_structure_on_resolvable_motives} 
The \emph{Chow weight structure} on the stable $\infty$-category $\DM^\kgl_{S,\res}$ of resolvable motives is defined as
\[
(\DM^\kgl_{S,\res,\geq 0}, \DM^\kgl_{S,\res,\leq 0}),
\]
where $\DM^\kgl_{S,\res,\geq 0}$ (resp. $\DM^\kgl_{S,\res,\leq 0}$) is the subcategory containing retracts of finite colimits (resp. limits) of perfect pure motives. 
\end{definition}

\begin{lemma}
The Chow weight structure is bounded. Furthermore, the heart of the Chow weight structure is $\Pure^\kgl_S$:
\[
\Pure^\kgl_S \simeq \DM^\kgl_{S,\res,\geq 0} \cap \DM^\kgl_{S,\res,\leq 0}.
\] 
\end{lemma}
\begin{proof}
The first claim is a general fact about weight structures generated by negatively self-orthogonal categories \cite[Remark~2.2.6]{elmanto-sosnilo}. The second claim follows from the fact that the heart of $\DM^\kgl_{S,\res}$ is the smallest retract-closed additive category containing $\Pure^\kgl_S$ (\cite[Remark~2.2.6]{elmanto-sosnilo}), so it is $\Pure^\kgl_S$ itself.
\end{proof}

Recall that for an additive $\infty$-category $\Ac$, Elmanto--Sosnilo denote by $\Ac^\fin$ the stable $\infty$-category generated by the essential image of the Yoneda-embedding in additive spectrum-valued presheaves $\Fun^\times(\Ac^\op, \Sp)$. This category is the free stable $\infty$-category generated by the additive category $\Ac$ (\cite[Remark~2.1.18]{elmanto-sosnilo}) in the sense that if $\Dc$ is a stable $\infty$-category, then precomposition along $\Ac \hook \Ac^\fin$ induces an equivalence
\begin{equation}\label{eq:UnivPropOf}
\Fun^\exa (\Ac^\fin, \Dc) \simeq \Fun^\times(\Ac,\Dc).
\end{equation}
Applying the recognition principle of Elmanto--Sosnilo \cite[Theorem~2.2.9]{elmanto-sosnilo} in the case $\Ac = \Pure_k$, we obtain the following result.

\begin{proposition}\label{prop:PureToRes}
The restricted Yoneda embedding $\DM^\kgl_{S,\res} \hook \Fun^\times(\Pure^\kgl_S, \Sp)$ induces an identification 
\[
\DM^\kgl_{S,\res} \simeq \Pure_S^{\kgl,\fin}.
\]
Thus, if $\Dc$ is a stable $\infty$-category, then there is a canonical identification
\[
\Fun^\exa (\DM^\kgl_{S,\res}, \Dc) \simeq \Fun^\times(\Pure^\kgl_S,\Dc).
\]
\end{proposition}

\begin{proof}
Indeed, $\DM^\kgl_{S,\res}$ is idempotent complete and the Chow weight structure is bounded. Thus \cite[Theorem~2.2.9(2)]{elmanto-sosnilo} implies that $\DM^\kgl_{S,\res} = \Pure_S^{\kgl,\fin}$.
\end{proof}

\begin{proof}[{Proof of \cref{thm:IndependenceOfWts}:}]
Since the Whitehead filtration $\tau_{\geq - \ast} \colon \spectra \rightarrow \Fil_\up(\spectra)$ is additive, so is the functor \
\[
\tau_{\geq -\ast} \circ E \colon \DM_{S}^{\kgl} \rightarrow \Fil_\up(\spectra). 
\]
Thus, this composite extends uniquely to the needed exact functor defined on all resolvable motives as a consequence of \cref{prop:PureToRes}. 
\end{proof}

\section{Comparison with the pole-order filtration on de Rham cohomology} 

In this section,  we compare our weight filtration with the filtration induced by pole orders on sheaves on differential forms with logarithmic singularities, introduced originally by Deligne \cite{deligne:1970,deligne:HodgeII} and studied further by Mokrane in positive characteristic \cite[\S 1]{mokrane:1993}.

\subsection{Logarithmic differential forms}

In this section we collect the properties of logarithmic differentials forms we will need. 

\begin{recollection}[{Logarithmic differentials forms}]
\label{recollection:logarithmic_differentials_forms}
By an \emph{sncd pair} $(X,D)$ over a scheme $S$, we mean a smooth $S$-scheme $X$ and a relative snc divisor $D$ on $X$. We denote by $\Omega^a_{X/S}(\log D)$ the sheaf of $a$-forms of $X$ with at worst logarithmic singularities along $D$ \cite[\S 3.1]{deligne:HodgeII}. We will refer to this as the \emph{sheaf of logarithmic $a$-forms} on $(X,D)$. Furthermore, we denote by 
\begin{equation}
    \Omega^a_{X/S,c}(\log D) := \Omega^a_{X/S}(\log D) (-D)
\end{equation}
the sheaf of \emph{compactly supported logarithmic $a$-forms} on $(X,D)$. Concretely, $\Omega^a_{X/S,c}(\log D)$ is the submodule of $\Omega^a_{X/S}(\log D)$ obtained by multiplying it by the ideal of the snc divisor $D$. The external derivative on (meromorphic) differential forms acts on logarithmic differential forms and logarithmic differential forms with compact support. Moreover, the wedge product of differential forms induces a perfect pairing of quasi-coherent sheaves
\begin{equation}
    \land \colon \Omega^a_{X/S}(\log D) \times \Omega^{n-a}_{X/S,c}(\log D) \to \Omega^{n}_{X/S,c}(\log D) = \Omega^n_{X/S},
\end{equation}
where $n$ is the dimension of $X$ \cite[Lemma~9.19]{esnault-viehweg}.
\end{recollection}

\begin{recollection}[{Residue and localization sequences}]
\label{recollection:residue_and_localization_sequences}
The sheaves of ordinary and compactly-supported logarithmic differential forms have the opposite functoriality in the following sense: if $D'$ is a relative snc divisor on $X$ that is contained in $D$, and $j \colon (X,D) \to (X,D')$ is the induced map of sncd pairs\footnote{A morphism $f \colon (X,D) \to (Y,E)$ of sncd pairs is a map of schemes $f \colon X \to Y$ such that $f^{-1} E \subset D$.}, then we obtain natural maps
\begin{equation}\label{eq:jPullDef}
    j^* \colon \Omega^a_{X/S}(\log D') \to \Omega^a_{X/S}(\log D)
\end{equation}
and
\begin{equation}\label{eq:jPushDef}
    j_* \colon \Omega^a_{X/S,c}(\log D) \to \Omega^a_{X/S,c}(\log D')
\end{equation}
from the natural inclusions of sheaves. Referring to maps $j$ of the above form as \emph{divisor enlargements}, we see that $\Omega^a_{X/S}(\log D)$ is contravariant and $\Omega^a_{X/S,c}(\log D)$ is covariant with respect to divisor enlargements. Furthermore, if $D_1$ is a component of $D$, we obtain the \emph{residue sequence}
\begin{equation}
\label{eq:res}
    0 \to \Omega^a_{X/S}(\log (D - D_1)) \xto{j^*} \Omega^a_{X/S}( \log D) \xto{\res} i_*\Omega^{a-1}_{D_1/S}(\log(D-D_1)\vert_{D_1}) \to 0
\end{equation}
and the \emph{localization sequence}
\begin{equation}
\label{eq:loc}
    0 \to \Omega^a_{X/S,c}(\log D) \xto{j_*} \Omega^a_{X/S,c}(\log (D-D_1)) \xto{i^*} i_*\Omega^a_{D_1/S,c}(\log(D-D_1)\vert_{D_1}) \to 0,
\end{equation}
where $\res$ is the residue morphism, and $i^*$ is the restriction of differential forms along $i \colon D_1 \hook X$. Both of these sequences are exact by \cite[\S 2]{esnault-viehweg}. 
\end{recollection}

We next record the naturality properties of the residue and localization sequences as lattice diagrams, a construction that, though perhaps unmotivated at this point, will be essential later for analyzing certain iterated fibers and cofibres of logarithmic cohomology complexes.

\begin{notation}[Labeling of the square]
\label{notation:labeling_of_the_square}
To keep track of certain exact sequences, it will be convenient to denote the square 1-category as $\mathrm{sq}^{1} \colonequals \Delta^{1} \times \Delta^{1}$ and to label its vertices using $-1, 0, \ast, 1$ as in the left hand side of
\[\begin{tikzcd}
	{-1} & 0 &\phantom{o}&\phantom{o}& a & b \\
	\ast & 1 &\phantom{o}&\phantom{o}& {0_{\acat}} & c
	\arrow[from=1-1, to=1-2]
	\arrow[from=1-1, to=2-1]
	\arrow[from=1-2, to=2-2]
	\arrow[from=1-5, to=1-6]
	\arrow[from=1-5, to=2-5]
	\arrow[from=1-6, to=2-6]
	\arrow[from=2-1, to=2-2]
	\arrow[from=2-5, to=2-6]
    \arrow[phantom, from=1-3, to=2-3, ""{name=LM, pos=.5, inner sep=0}]
	\arrow[phantom, from=1-4, to=2-4, ""{name=RM, pos=.5, inner sep=0}]
	\arrow[decorate,
	       decoration={snake, amplitude=0.2mm, segment length=2mm},
	       ->,
	       from={LM}, to=RM]
\end{tikzcd}.
\]
If $a \rightarrow b \rightarrow c$ is a null-sequence in an additive $\infty$-category $\acat$, then the \emph{associated square} is a functor $p \colon \mathrm{sq}^{1} \rightarrow \acat$ given by $p(-1) = a$, $p(0) = b$, $p(1) = c$ and $p(\ast) = 0_{\acat}$, as on the right hand side above. 
\end{notation}

\begin{construction}[Residue and localization sequences diagrams]
\label{construction:residue_and_localization_sequence_diagrams}
Let $D = D_1 + \cdots + D_r$ be a decomposition of $D$ into its components. We describe diagrams 
\begin{equation}
\mathfrak{res}, \mathfrak{loc} \colon \mathrm{sq}^{\times r} \to \QCoh^\heart(X)   
\end{equation}
valued in the abelian category of quasi-coherent sheaves on $X$, where $\mathrm{sq}$ is the square category with vertices labeled as in \cref{notation:labeling_of_the_square}. Their values at $\underline \alpha = (\alpha_1, \dots, \alpha_r)$ are given by, respectively, 
\[
   \mathfrak{res}(\underline \alpha) \colonequals \begin{cases}
    \Omega^{a - n^+(\underline \alpha)}_{D^+_{\underline \alpha} / S}(\log (D- D^{\pm}(\underline \alpha))) & \text{if } \alpha_{i} \neq \ast \text{ for all } i \\
    0 & \text{otherwise} 
\end{cases} 
\]
and
\[
   \mathfrak{loc}(\underline \alpha) \colonequals \begin{cases}
    \Omega^a_{D^+_{\underline \alpha} / S, c}(\log (D- D^{\geq}(\underline \alpha))) & \text{if } \alpha_{i} \neq \ast \text{ for all } i \\
    0 & \text{otherwise}
\end{cases},
\]
where
\begin{enumerate}
\item $n^+(\underline \alpha)$ is the number of indices $i$ such that $\alpha_i = 1$,
\item $D^+_{\underline \alpha} = \bigcap_{\alpha_i = 1} D_i$, 
\item $D^{\pm}(\underline \alpha) = \sum_{\alpha_i \not = 0} D_i$ and 
\item $D^{\geq}(\underline \alpha) = \sum_{\alpha_i \geq 0} D_i$.
\end{enumerate}
The morphisms are such that for each $1 \leq k \leq r$, the hypercolumns where we fix all coordinates $\alpha_{j}$ with $j \neq k$ correspond, respectively, to the residue exact sequence of (\ref{eq:res}) or the localization exact sequence of (\ref{eq:loc}) associated to $D_{i}$.  
\end{construction}

\begin{remark}
In the context of \cref{construction:residue_and_localization_sequence_diagrams}, since the diagram takes values in an abelian category, whether the composite of two morphisms is zero or not is a property rather than additional data. Thus, if we were only interested in diagrams valued in additive $1$-categories, instead of the square category $\mathrm{sq}$ we could have alternatively used its full subcategory spanned by $\{ -1, 0, 1 \}$. 

However, we will later want to consider diagrams valued in a stable $\infty$-category, in which case the nullhomotopy of a composite is additional data. The somewhat elaborate setup of \cref{notation:labeling_of_the_square} is designed to keep track of these null-homotopies as part of a functor out of a square category. 
\end{remark}

Let $(X,D)$ be an sncd pair over $S$. We denote by $P_* \Omega^a_{X/S}(\log D)$ the filtered sheaf obtained by equipping the sheaf of logarithmic $a$-forms with the \textit{filtration by pole order} \cite[\S 1]{mokrane:1993}. For an integer $j \geq 0$, the subsheaf
\[
P_j \Omega^a_{X/S}(\log D) \subset \Omega^a_{X/S}(\log D)
\]
consists of those logarithmic differential forms that, locally on $X$, have poles along at most 
$j$ irreducible components of $D$. Concretely, if \'etale--locally we choose parameters 
$t_1,\dots,t_r$ defining the components of $D$, then 
$P_j \Omega^a_{X/S}(\log D)$ is generated by forms of the form
\(
\tfrac{dt_{i_1}}{t_{i_1}} \wedge \cdots \wedge \tfrac{dt_{i_k}}{t_{i_k}} \wedge \omega,
\)
where $k \le j$ and $\omega$ is a regular $(a-k)$-form. This filtration is multiplicative and compatible with pullbacks.

By inspection, we observe that the residue sequence of (\ref{eq:res}) is compatible with this filtration in the following sense.

\begin{lemma}
\label{lem:filres}
The residue sequence enhances into an exact sequence
\begin{equation}
0 \to P_*\Omega^a_{X/S}(\log (D - D_1)) \xto{j^*} P_* \Omega^a_{X/S}( \log D) \xto{\res} i_*P_{*-1}\Omega^{a-1}_{D_1/S}(\log(D-D_1)\vert_{D_1})\to 0
\end{equation}
of filtered quasi-coherent sheaves. \qed
\end{lemma}

Moreover, the graded pieces of this filtration admit an explicit formula. 

\begin{lemma}
The residue morphisms induce a canonical isomorphism
\begin{equation}
    \Gr_j^P \Omega^a_{X/S}( \log D) = \tau_{j*} \Omega^{a-j}_{D^{(j)}/S},
\end{equation}
where $\tau_j \colon D^{(j)} \to X$ is the natural map from the disjoint union $D^{(j)}$ of $j$-fold intersections of components of $D$. 
\end{lemma}

\begin{proof}
This is the level-wise version of \cite[Equation~(1.1.2)]{mokrane:1993}.
\end{proof}
Using this, we make the following observation that will be useful shortly.

\begin{lemma}
\label{lem:filres_graded_split}
The filtered exact sequence of \cref{lem:filres} induces split exact sequences on graded pieces.
\end{lemma}
\begin{proof}
Denote $D' = D-D_1$. Then the $j$th graded piece of the sequence associated to the exact sequence in \cref{lem:filres} is
\begin{equation}
0 \to \tau_{j*} \Omega^{a-j}_{D'^{(j)}/S} \to \tau_{j*} \Omega^{a-j}_{D^{(j)}/S} \to \Omega^{a-j}_{D_1 \cap D'^{(j-1)}/S} \to 0.
\end{equation}
The claim follows from the fact that $D^{(j)} = D'^{(j)} \coprod (D_1 \cap D'^{(j-1)})$ and that the maps induced by inclusions of components.
\end{proof}

\subsection{de Rham and Hodge cohomology as motivic spectra} 
\label{subsection:de_rham_and_hodge_cohomology_as_motivic_spectra}

In this subsection, we construct motivic spectra representing motivic filtered algebraic $K$-theory, motivic cohomology and variants of de Rham cohomology. We work over an arbitrary derived scheme $S$, and the motivic filtration on algebraic $K$-theory we employ is the one of \cite{elmanto-morrow,bouis:2024}. 

Our construction is a slight modification of \cite[Construction~6.5]{AHI:atiyah}, and we follow the notation used therein. For more details, see \cite[Appendix C]{annala-shin}. 

\begin{cons}[Motivic spectra from modules over oriented cohomology theories]\label{cons:MSFromModules}
Let $S$ be a derived scheme, $\Pc = \Pc_{\Nis}(\Sm_S; \spectra)$ be the category of Nisnevich sheaves of spectra on $\Sm_S$, and let $E \in \CAlg(\Gr_\Zb(\Pc))$ be an oriented commutative algebra in the sense of \cite[Appendix~C]{annala-shin}. We construct a lax symmetric monoidal functor 
\begin{equation}
\forgetful \colon \Mod_E^\mathrm{pbf}(\Gr_\Zb(\Pc)) \to \MS_S. 
\end{equation}
We first observe that we have a symmetric monoidal left adjoint
\begin{equation}\label{eq:SomeShit}
    \Pc \to \Mod_E^\mathrm{pbf}(\Gr_\Zb(\Pc))
\end{equation}
obtained by left Kan extending the map that sends $X \in \Sm_S$ to the free $E$-module generated by $\underline X$ placed in graded degree zero and enforcing graded projective bundle formula. As a consequence of projective bundle formula, $\Mod_E^\mathrm{pbf}(\Gr_\Zb(\Pc))$ satisfies smooth blowup excision \cite[Lemma~3.3.4]{annala-iwasa:MotSp}. Moreover, the endomorphism $\Tbf \otimes -$ is equivalent to degree shift by one, and is therefore invertible. Thus, by the universal property of $\otimes$-inversion, the functor of (\ref{eq:SomeShit}) factors uniquely through a symmetric monoidal left adjoint
\[
- \otimes E \colon \MS_S \to \Mod_E^\mathrm{pbf}(\Gr_\Zb(\Pc)).
\]
Its right adjoint $\forgetful \colon \Mod_E^\mathrm{pbf}(\Gr_\Zb(\Pc)) \to \MS_S$ is the desired lax symmetric monoidal functor. 
\end{cons}

\begin{remark}
\label{remark:formula_for_sheaf_defined_by_our_motivic_spectra}
The construction of $\forgetful$ as a right adjoint implies that if $M$ is an $E$-module, then $\forgetful(M)$ represents $M$ in the sense that 
\begin{equation}
    \map_{\MS_S}(\Sigma_\Tbf^{\infty-a}X_+, \forgetful(M)) \simeq  M_a(X)
\end{equation}
for any $X \in \Sm_{S}$. 
\end{remark}

\begin{remark}
Note that since $\forgetful$ is lax symmetric monoidal, it has a canonical lift 
\[
\forgetful \colon \Mod_E^\mathrm{pbf}(\Gr_\Zb(\Pc)) \to \Mod_{\forgetful(E)}\MS_S. 
\]
to modules over the image of the monoidal unit. 
\end{remark}

We first use \cref{cons:MSFromModules} to construct $\kgl$ itself. As above, $S$ is a derived scheme and $\Pc = \Pc_{\Nis}(\Sm_S; \spectra)$ is the category of Nisnevich sheaves of spectra on $\Sm_S$. In a recent article constructing a motivic filtration on the algebraic $K$-theory of arbitrary schemes \cite{bouis:2024}, Bouis refines algebraic $K$-theory to a filtered sheaf 
\[
F^{*}_{\mot} K \in \CAlg(\Fil^{\down}(\Pc)).
\]

\begin{definition}
\label{definition:kgl_and_hmot_as_motivic_spectra}
We write $\kgl$ for the image of  $\Fil^*_\mathrm{mot} K$ under the composite 
\[
\begin{tikzcd}
	{\CAlg(\Fil^{\down}(\Pc))} & {\CAlg(\Gr_{\mathbb{Z}}(\Pc))} & {\CAlg(\MS_S)}
	\arrow["{\mathrm{U}}", from=1-1, to=1-2]
	\arrow["\forgetful", from=1-2, to=1-3]
\end{tikzcd}
\]
where the first arrow is the forgetful functor from filtered to graded objects and the second one is given by \cref{cons:MSFromModules}. We write $\mathrm{H} \Zb$ for the image of $\Fil^*_\mathrm{mot} K$ under the composite 
\[
\begin{tikzcd}
	{\CAlg(\Fil^{\down}(\Pc))} & {\CAlg(\Gr_{\mathbb{Z}}(\Pc))} & {\CAlg(\MS_S)}
	\arrow["{\mathrm{gr}}", from=1-1, to=1-2]
	\arrow["\forgetful", from=1-2, to=1-3]
\end{tikzcd}, 
\]
where $\mathrm{gr}$ is the associated graded object functor. We refer to these as, respectively, the \emph{effective algebraic $K$-theory} and \emph{motivic cohomology} motivic spectra. 
\end{definition}

\begin{remark}
The canonical natural transformation $U \rightarrow \mathrm{gr}$ given by projecting from a filtered object onto its associated graded provides a morphism $\kgl \rightarrow \mathrm{H} \Zb$ of commutative algebras in motivic spectra. 
\end{remark}

\begin{remark}
Note that in the case that is most relevant to this article, namely when $S$ is an affine Dedekind scheme, $F^*_\mathrm{mot}K$ coincides with the slice filtration of Voevodsky, see \cite[Remark~6.9]{bouis:2024} and \cite[Theorem~A]{bouis-kundu}. In particular, for smooth schemes over affine Dedekind schemes, the motivic spectra of \cref{definition:kgl_and_hmot_as_motivic_spectra} coincide with the classical ones.
\end{remark}

We move on to  motivic spectra coming from Hodge theory. There is a morphism 
\begin{equation}
    F^*_\mot K \to \Fil_\HKR^* \mathrm{HC}^-(-) \to F^*_\mathrm{HKR} \mathrm{HC}^-(-/S) \in \CAlg(\Fil^{\down}(\Pc)) 
\end{equation}
where the first map is constructed by Bouis, see \cite[\S 2.1--\S 2.3,\S 3.3]{bouis:2024}, and the second map is induced by the collapse formula \cite[Proposition~2.7]{antieau:gaussmanin} by taking the fixed points of the filtered circle action \cite[Theorem~8.17]{antieau:crystallization}. By passing to the associated graded objects and shearing \cite[Proposition~3.3.4]{raksit:2020}, we then obtain a morphism
\begin{equation}
    \Zb(*)^\mot \to U(F^*_\Hod\Omega^\bullet_{-/S}) \in \CAlg(\Gr(\Pc)), 
\end{equation}
from the graded algebra of motivic sheaves to the underlying graded algebra of the Hodge filtered de Rham cohomology over $S$.

\begin{definition}
\label{definition:hodge_filtered_de_rham_cohomology_and_its_cousins}
The \emph{Hodge-filtered de Rham cohomology} filtered motivic spectrum 
\[
\mathrm{H}^{\ast}\dR(S) \in \CAlg(\Fil^\down(\Mod_{\mathrm{H}\Zb}(\MS_S)))
\]
is the filtered algebra obtained from 
\[
\left[ \cdots \to  U(F^{*-1}_\Hod\Omega^\bullet_{-/S}) \to U(F^{*}_\Hod\Omega^\bullet_{-/S}) \to \cdots \right] \in \CAlg(\Fil^\down(\Mod_{\Zb(*)^\mot}(\Gr(\Pc)))), 
\]
where the structure morphisms are given by the structure maps of the Hodge filtration, by applying the functor $F$ from \cref{cons:MSFromModules} level-wise.
\end{definition}

Having defined a motivic spectrum representing Hodge filtered de Rham cohomology, we can define motivic spectra representing Hodge and de Rham cohomology. 

\begin{definition}
We denote by
\[
\Hod^{\ast}(S) \colonequals \mathrm{gr}(\mathrm{H}^{\ast} \dR(S)) \in \CAlg(\Gr(\Mod_{\mathrm{H}\Zb}(\MS_S)))
\] 
the \emph{Hodge cohomology} graded motivic spectrum. 
\end{definition}

\begin{definition}
We define
\[
\dR(S) \colonequals \varinjlim \mathrm{H}^{\ast} \dR(S) \in \CAlg(\Mod_{\mathrm{H}\Zb}(\MS_S)) 
\] 
as the \emph{de Rham cohomology} motivic spectrum.
\end{definition}

We note that by restriction of scalars along $\kgl \to \mathrm{H}\Zb$, we obtain (filtered, graded) algebras in $\DM^\kgl_S$.

\begin{remark}
\label{remark:hdr_and_cousins_represent_de_rham_cohomology}
The terminology of \cref{definition:hodge_filtered_de_rham_cohomology_and_its_cousins} is justified in the following way: If $X$ is a smooth $S$-scheme, then \cref{remark:formula_for_sheaf_defined_by_our_motivic_spectra} yields an identification of filtered $\ZZ$-modules 
\begin{equation}
\label{equation:mapping_spectrum_into_hodge_filtered_de_rham_cohomology}
\map_{\MS_{S}}(\Sigma^{\infty}_{\Tbf} X_{+}, \Sigma^{a}_{\Tbf} \mathrm{H}^{\ast} \dR(S)) \simeq \R\Gamma(X, \Omega^{\geq \ast+a}_{-/S}[2a]).
\end{equation}
In other words, $\mathrm{H}^{\ast} \dR(S)$ represents Hodge-filtered de Rham cohomology of smooth $S$-schemes. Note that by allowing $X$ to vary, we obtain an equivalence 
\[
\Omega^{\infty}_{\Tbf} \Sigma^{a}_{\Tbf} \mathrm{H}^{\ast} \dR(S) \simeq \Omega^{\geq +a}_{-/S}[2a]
\]
of filtered Nisnevich sheaves of spectra. In an analogous sense, $\dR(S)$ and $\mathrm{Hdg}^{\ast}(S)$ represent, respectively, de Rham cohomology and Hodge cohomology. 
\end{remark}

\begin{remark}
Suppose that $X$ is a smooth $S$-scheme. Since $\Sigma^{\infty}_{\Tbf}(-)_{+} \colon \Sm_{S} \rightarrow \MS_{S}$ is symmetric monoidal, $\Sigma^{\infty}_{\Tbf} X_{+}$ has an induced cocommutative coalgebra structure and consequently 
\[
\map_{\MS_{S}}(\Sigma^{\infty}_{\Tbf} X_{+}, \mathrm{H}^{\ast} \dR(S)) \simeq \mathrm{filmap}_{\Fil_\up(\MS_{S})}(
\Sigma^{\infty}_{\Tbf} X_{+}
\{ 0 \} , \mathrm{H}^{\ast} \dR(S)), 
\] 
where on the right we take the mapping spectrum and allow the target to vary, while on the left $(-)\{0\}$ denotes the free filtered motivic spectrum generated in degree zero and $\mathrm{filmap}$ the internal filtered mapping spectrum, has a canonical structure of a commutative algebra in filtered spectra. This is compatible with \cref{remark:hdr_and_cousins_represent_de_rham_cohomology} in the sense that when $a = 0$, (\ref{equation:mapping_spectrum_into_hodge_filtered_de_rham_cohomology}) is an equivalence of commutative algebras. 
\end{remark}

\subsection{Statement of the result}

In this section, we introduce the relevant notation and state \cref{thm:DeligneCompare}, which identifies the weight filtration on logarithmic de Rham cohomology obtained from \cref{thm:IndependenceOfWts} with the décalage of the pole order filtration on global sections. We prove the result in the next section. 

\begin{definition}
Let $(X,D)$ be a strict normal crossing divisor pair of dimension $n$ over $S$. The  \emph{logarithmic motive} is the total fibre 
\[
    \Mbf_{\kgl}(X,D) \colonequals \tfib_{I \subset [r]} (\Mbf_{\kgl}(D_I)(-n)^\vee) \in \DM^\kgl_{S,\res}
\]
of the hypercube given by the dual of the Tate-twisted diagram of the canonical inclusions between $D_I = \bigcap_{i \in I} D_i$. 
\end{definition}

We note that by construction, the logarithmic motive is resolvable. 

\begin{remark}[{Logarithmic motive and the open complement}]
\label{remark:logarithmic_motives_in_terms_of_the_open_complement}
As a consequence of the Morel--Voevodsky purity theorem \cite{morel:1999}, the logarithmic motive can be described as 
\[
\Mbf_{\kgl}(X,D) \simeq \Mbf^{\Abf^1}_{\kgl}(U) 
\]
the \emph{$\Abf^{1}$-invariant} motive\footnote{Note that since our motives are $\kgl$-linear, the motive of a smooth projective variety is automatically $\mathbf{A}^{1}$-invariant.} of $U = X \setminus D$, see \cite[Lemma~6.24]{AHI:atiyah}. In particular, it is an invariant of $U$.
\end{remark}

\begin{notation}
Let $(X, D)$ be a strict normal crossing divisor pair over $S$. We write 
\[
W_*\dR(S)(\Mbf_{\kgl}(X,D)) \in \Fil_{\up}(\D(S)) 
\]
for the filtered object obtained by applying \cref{thm:IndependenceOfWts} to the functor 
\[
\dR(S)(-) \colonequals \map_{\DM^{\kgl}_{S}}(-, \dR(S)) \colon (\DM^{\kgl}_{S, \res})^{op} \rightarrow \D(S)
\]
and the standard t-structure on $\D(S)$. Here, $\dR(S)$ stands for the de Rham cohomology motivic spectrum of \cref{definition:hodge_filtered_de_rham_cohomology_and_its_cousins}. Analogously, we write 
\begin{align*}
W_*\mathrm{H}^{\ast} \dR(S)(\Mbf_{\kgl}(X,D)) \in \Fil_\up(\DF(S))  \\ 
W_*\Hod^{\ast}(S)(\Mbf_{\kgl}(X,D)) \in \Fil_\up(\DGr(S)) 
\end{align*}
for the objects obtained, respectively, using Hodge filtered de Rham cohomology and the degreewise Whitehead filtration\footnote{Beware that the choice of the t-structure is important. The $\infty$-category of filtered objects supports a plethora of interesting t-structures (for example, the Beilison t-structure which is used to define décalage).  However, it is the degreewise Whitehead t-structure on Hodge filtered de Rham cohomology which we compare to the pole order filtration.} on $\DF(S)$, and Hodge cohomology and the degreewise Whitehead filtration on $\DGr(S)$. 
\end{notation}

\begin{notation}
\label{notation:d_filtered_cohomology}
Let $(X,D)$ be an sncd pair over a scheme $S$. We write 
\begin{align}
D_* \R\Gamma(X; \Omega^\bullet_{X/S}(\log D)) &:= \Dec(P)_*\R\Gamma(X; \Omega^\bullet_{X/S}(\log D)) \in \Fil_\up(\D(S)) \\
D_* \R\Gamma(X; \Omega^{\geq *}_{X/S}(\log D)) &:= \Dec(P)_*\R\Gamma(X; \Omega^{\geq *}_{X/S}(\log D)) \in \Fil_\up(\DF(S)) \\
D_* \R\Gamma(X; \Omega^*_{X/S}(\log D)) &:= \Dec(P)_*\R\Gamma(X; \Omega^*_{X/S}(\log D)) \in \Fil_\up(\DGr(S)) 
\end{align}
for the filtrations on logarithmic de Rham cohomology (resp. logarithmic Hodge filtered de Rham cohomology, logarithmic Hodge cohomology) obtained by taking the filtration induced by the pole order filtration $P_*$ on global sections, still denoted by $P_*$, and where for an increasing filtration $F_*$
\begin{equation}
\Dec(F)_n = \varinjlim (\tau^B_{\geq -n} (F_*))
\end{equation}
is the increasing variant of décalage from \cite[{Construction 4.5}]{antieau:decalage}.
\end{notation}

\begin{remark}
For further background and the relationship between décalage and spectral sequences, see \cite{antieau:decalage}. 
\end{remark}

Having set up the notation, we now state the comparison result, which we prove in the next section.  
  
\begin{theorem}
\label{thm:DeligneCompare}
Let $(X,D)$ be a projective sncd pair over an affine Dedekind scheme $S$. Then, there exist equivalences 
\begin{align}
D_* \R\Gamma(X; \Omega^\bullet_{X/S}(\log D)) &\simeq W_*\dR(S)(\Mbf_\kgl(X,D)) \\
D_* \R\Gamma(X; \Omega^{\geq *}_{X/S}(\log D)) &\simeq  W_* \mathrm{H}^{*} \dR(S)(\Mbf_\kgl(X,D)) \\
D_* \R\Gamma(X; \Omega^*_{X/S}(\log D)) &\simeq W_*\Hod^*(S)(\Mbf_\kgl(X,D))
\end{align}
that are contravariantly functorial in divisor enlargements of projective sncd pairs. In other words, the décalaged pole-order filtration coincides with the weight filtration of \cref{thm:IndependenceOfWts} obtained from the (degreewise) Whitehead filtration. 
\end{theorem}

\begin{remark}[Naturality and multiplicativity of the comparison equivalence]\label{rem:OurFailures}
The proof strategy adopted in this article does not directly yield functoriality or multiplicativity of the equivalences of \cref{thm:DeligneCompare} with respect to arbitrary morphisms of sncd pairs. Indeed, the argument proceeds by identifying both sides of the comparison with the total cofibre of a common hypercube diagram, and the functoriality of this construction is limited by the functoriality properties of the underlying hypercube.

In \cref{quest:ResMots}, we outline an alternative approach based on describing the category of resolvable motives in a geometric manner, analogous to the construction of $\MS_S$ or $\mathrm{logSH}_S$ from schemes, but using only projective sncd pairs. 
This is motivated by the fact that the décalaged pole-order filtration is not a sheaf, making it impossible to enhance it into a functor out of $\mathrm{logSH}_S$. 
A purely geometric presentation of resolvable motives would allow one to package the décalaged pole-order filtration as a functor out of resolvable motives, and thereby obtain functorial comparison equivalences for maps of sncd pairs as a consequence of \cref{prop:PureToRes}.
\end{remark}

\begin{corollary}
Let $S$ be an affine Dedekind scheme. Then logarithmic de Rham cohomology (resp. logarithmic Hodge filtered de Rham cohomology, logarithmic Hodge cohomology) of projective sncd pairs over $S$, together with the décalaged pole-order filtration, is an intrinsic invariant of $U$ which does not depend on the choice of a good compactification. 
\end{corollary}

\begin{proof}
Since the right hand side only depends on the logarthmic motive, this is immediate from \cref{remark:logarithmic_motives_in_terms_of_the_open_complement}. 
\end{proof}

\begin{remark}
If we forget the filtrations, the first equivalence recovers  \cite[Proposition~6.27]{AHI:atiyah}. In the case of Hodge filtered de Rham and Hodge cohomology, \cref{thm:DeligneCompare} is new even after forgetting the filtrations. 
\end{remark}

\subsection{Proof of the comparison} 

In this section, we prove \cref{thm:DeligneCompare}. We begin by showing that the $D_*$-filtration on the relevant logarithmic cohomologies can be computed as a total cofibre of a cube involving only Whitehead filtrations. To do so, we make the following observations. 

\begin{lemma}
\label{lem:DecTurnsInsToWhitehead}
If the divisor $D$ is empty, then the $D_*$-filtration on logarithmic de Rham cohomology (resp. logarithmic Hodge filtered de Rham cohomology, resp. logarithmic Hodge cohomology) coincides with the (degreewise, negatively indexed) Whitehead filtration. 
\end{lemma}

\begin{proof}
If the logarithmic structure is trivial, then the $P_*$-filtration coincides with the filtration obtained by free insertion to filtered degree 0. The claim follows from the fact that décalage turns this filtration into the Whitehead filtration.
\end{proof}

\begin{lemma}
\label{lem:decfilres}
Applying décalage to the filtered exact sequence of \cref{lem:filres} induces a cofibre sequence
\begin{equation}
D_* \R\Gamma(X; \Omega^{\geq a}_{X/S}(\log (D-D_1))) \to D_* \R\Gamma(X; \Omega^{\geq a}_{X/S}(\log D)) \to \big(D_{*-2} \R\Gamma(X; \Omega^{\geq a - 1}_{D_1/S}(\log (D-D_1) \vert_{D_1}))\big)[-1]
\end{equation}
in $\Fil^{\down}(\D(S))$. Analogous cofibre sequences in the filtered derived category are induced also for logarithmic de Rham cohomology and for logarithmic Hodge cohomology.
\end{lemma}
\begin{proof}
By \cref{lem:filres}, we have a level-wise exact sequence
\begin{equation}
0 \to P_* \Omega^{\geq a}_{X/S} (\log (D-D_1)) \to P_* \Omega^{\geq a}_{X/S} (\log D) \to P_{*-1} \Omega^{\geq a-1}_{X/S} (\log (D-D_1)\vert_{D_1})[-1] \to 0    
\end{equation}
of chain complexes of filtered sheaves. Thus, applying the global sections functor $\R\Gamma$ to it, we obtain a cofibre sequence in $\DF(S)$. By \cref{lem:filres_graded_split}, this cofibre sequence induces split cofibre sequences on graded pieces. Thus, applying the Beilinson truncation functor $\tau^B_{\geq n}$ (see \cite[Definition~3.24]{antieau:decalage}) to it results in a cofibre sequence for all $n \in \Zb$. As décalage is defined using Beilinson truncations \cite[Construction~4.5]{antieau:decalage}, the claim for logarithmic Hodge filtered de Rham cohomology follows. Note that the residue morphism shifts the $D_*$ filtration by $-2$, since if $F_* C$ is a filtered complex, then
\begin{equation}
\Dec(F_*)(C[-1]) = (F_{*-1}C)[-1].
\end{equation}
The same proof strategy works for logarithmic de Rham and logarithmic Hodge cohomology.
\end{proof}

We will also need the following filtered enhancement of Poincaré duality for logarithmic de Rham cohomology, originally proven by Tsuji \cite{tsuji:1999}. Before stating and proving the result, we introduce the following notation.

\begin{notation}
By $\Fc\{i\} \in \DF(S)$ (resp. $\Fc(i) \in \DGr(S)$) we mean the object obtained by freely placing $\Fc$ into filtered (resp. graded) degree $i$ if $\Fc$ was not already filtered (resp. graded), or the object obtained by shifting the filtration (resp. grading) by $i$ degrees if $\Fc$ was already filtered (resp. graded).
\end{notation}

\begin{notation}
\label{notation:sncd_category}
For a smooth $S$-scheme $X$, we let $\SNCD_{X/S}$ be the category 
\begin{enumerate}
\item whose objects are relative sncd pairs $(X,D)$ with underlying scheme $X$,
\item and where the maps are divisor enlargements $j \colon (X,D) \to (X,D')$, where $D' \subset D$. 
\end{enumerate}
Note that variants of logarithmic de Rham cohomology induce functors $\SNCD_{X/S}^\op \to \D(S)$ whose structure morphisms are pullbacks $j^*$, whereas the compactly supported variants induce functors $\SNCD_{X/S} \to \D(S)$ whose structure morphisms are pushforwards $j_*$, see \cref{recollection:residue_and_localization_sequences}. 
\end{notation}

\begin{theorem}[Poincaré duality]
\label{thm:PoincarePlus}
Let $S$ be a scheme and let $X$ be a smooth projective $S$-scheme of relative dimension $n$. Then the trace map from Serre duality $\tr \colon \R\Gamma(X;\Omega^{n}_{X/S,c}(\log D)[-n]) \to \Oc_S[-2n]$ lifts to a trace map
\begin{equation}
    \tr \colon \R\Gamma(X;\Omega^{\geq *}_{X/S,c}(\log D)) \to \Oc_S[-2n]\{n\} \in \DF(S).
\end{equation}
Together, these trace maps induce perfect pairings
\begin{align}
    \R\Gamma(X;\Omega^\bullet_{X/S}(\log D)) \times \R\Gamma(X;\Omega^\bullet_{X/S,c}(\log D)) &\xto{\land} \R\Gamma(X;\Omega^\bullet_{X/S,c}(\log D)) \xto{\tr} \Oc_S[-2n] \in \D(S) \label{eq:dRpairing}\\
    \R\Gamma(X;\Omega^{\geq *}_{X/S}(\log D)) \times \R\Gamma(X;\Omega^{\geq *}_{X/S,c}(\log D)) &\xto{\land} \R\Gamma(X;\Omega^{\geq *}_{X/S,c}(\log D)) \xto{\tr} \Oc_S[-2n]\{n\} \in \DF(S) \label{eq:HodFilpairing}\\
    \R\Gamma(X;\Omega^*_{X/S}(\log D)) \times \R\Gamma(X;\Omega^{*}_{X/S,c}(\log D)) &\xto{\land} \R\Gamma(X;\Omega^*_{X/S,c}(\log D)) \xto{\tr} \Oc_S[-2n](n) \in \DGr(S) \label{eq:Hodpairing}.
\end{align}
Moreover, each of these pairings admits a canonical enhancement that is bivariantly functorial with respect to divisor enlargements, yielding a functor
\[
\Tw(\SNCD_{X/S}) \to \D^{\Delta^1},
\qquad
j \longmapsto \bigl((\omega,\rho) \mapsto \tr(j^*\omega \wedge \rho)\bigr),
\]
where $\D$ denotes $\D(S)$, $\DF(S)$, or $\DGr(S)$, respectively. 
In particular, they induce natural duality equivalences
\begin{equation}
\R\Gamma(X;\Omega^\bullet_{X/S}(\log D))
\xrightarrow{\simeq}
\R\Gamma (X;\Omega^\bullet_{X/S,c}(\log D))^\vee
\colon
\SNCD_{X/S}^{\op} \to \D.
\end{equation}
\end{theorem}

\begin{proof}
The first claim is \cite[Proposition~3.1]{tsuji:1999}.\footnote{To be more precise, Tsuji states Proposition~3.1 only in the case where $S$ is a spectrum of a discrete valuation ring. We suspect that the proof of Proposition~3.1, given in \cite[\S 4]{tsuji:1999}, works over a general base scheme, but for compleneteness we supply here an alternative argument, by showing that the first claim also follows from Atiyah duality of \cite[Corollary~5.15]{AHI:atiyah}. Indeed, we claim the desired trace map is the composition 
\[
\R\Gamma(X;\Omega^{\geq *}_{X/S,c}(\log D)) \xto{j_*} \R\Gamma(X;\Omega^{\geq *}_{X/S,c}) \to \Oc_S[-2n]\{n\}
\]
where the second map is the pullback along the Tate-twisted dual of the structure map $\mathbf{1}_S(n) \to \Mbf_\kgl(X)$, at least up to multiplying with an invertible global function on $X$, which can easily be undone as the external differential $d$ vanishes on such functions by \cite[Tag~0G8H]{stacks}. 

It suffices to show that, after we pass to graded pieces (i.e., to Hodge cohomology), we recover the trace map of Grothendieck--Serre duality. We can ignore $j_*$ as is the identity for $\Omega^n_{X/S,c}(\log (-))$. As $\Hod^*(S) \colon \DM^\kgl_S \to \DGr(S)$ is symmetric monoidal, our trace morphism is defined by taking duals with respect to a perfect pairing, so the claim follows from the uniqueness of trace up to a multiplication by an invertible function \cite[Tag~0G8I]{stacks}. 
} The perfectness of the pairing of (\ref{eq:Hodpairing}) is Grothendieck--Serre duality. Because both filtered complexes in (\ref{eq:HodFilpairing}) are complete with respect to the filtration, and because passing to graded pieces is symmetric monoidal, perfectness of (\ref{eq:HodFilpairing}) follows from that of (\ref{eq:Hodpairing}). The perfectness of (\ref{eq:dRpairing}) then follows by forgetting the filtration. The fact that the bivariant enhancement of the pairing induces a functorial comparison is \cref{prop:functorial_duals}.
%
\end{proof}

Using Poincaré duality, we will derive an explicit formula for the $D_*$-filtration in terms of the Whitehead filtration. To do so, we will make use of a following categorical construction on cubes in stable $\infty$-categories. 

\begin{recollection}
As in \cref{notation:labeling_of_the_square}, we denote the square category by $\sq = \Delta^{1} \times \Delta^{1}$ and label its vertices using $\{ -1, 0, \ast, 1 \}$, with $-1$ and $1$ corresponding to, respectively, the initial and final vertices. We recall from \cite[Definition 1.1.1.4]{HA} that if $\ccat$ is a stable $\infty$-category, then a diagram $P \colon \sq \rightarrow \ccat$ is a \emph{cofibre sequence} if it is a pushout and $P(\ast) = 0_{\ccat}$. 
\end{recollection}

\begin{definition}
\label{definition:hypercolumn_cofibre_diagram} 
Let $\ccat$ be a stable $\infty$-category. We say that a functor 
\[
P \colon \sq^{\times r} \rightarrow \ccat 
\]
is \emph{hypercolumn cofibre} if its hypercolumns are cofibre sequences; that is, if for every $1 \leq i \leq r$, every choice of vertices $\beta_{k}$ of the square for $k \neq i$, the resulting functor 
\[
P(\beta_{1}, \ldots, \beta_{i-1}, -, \beta_{i+1}, \ldots, \beta_{r}) \colon \sq \rightarrow \ccat 
\]
is a cofibre sequence. 
\end{definition}

\begin{example}
\label{example:hypercolumn_cofibre_diagram_from_residue_sequences}
By equipping the residue diagram of \cref{construction:residue_and_localization_sequence_diagrams} with the pole-order filtration, considering it as valued in the filtered derived $\infty$-category and taking décalage, we obtain a diagram
\[
\mathrm{sq}^{\times r} \to \Fil(\D(S)), 
\]
of $D_*$-filtered logarithmic de Rham cohomology. Here, $\mathrm{sq} = \Delta^{1} \times \Delta^{1}$ is the square category as in \cref{notation:labeling_of_the_square}. As a consequence of \cref{lem:decfilres}, every hypercolumn of this diagram is a cofibre sequence. 

An analogous construction applied to logarithmic Hodge filtered de Rham and logarithmic Hodge cohomology yields diagrams $\mathrm{sq}^{\times r} \to \Fil(\DF(S))$ and $\mathrm{sq}^{\times r} \to \Fil(\DGr(S))$. These also have hypercolumns given by cofibre sequences. 
\end{example}

\begin{remark}
If $\ccat$ is a stable $\infty$-category, then the full subcategory 
\[
\Fun^{\mathrm{hpcofib}}(\sq^{\times r}, \ccat) \subseteq \Fun(\sq^{\times r}, \ccat) 
\]
spanned by hypercolumn cofibre diagrams is a thick subcategory. In particular, it is stable. 
\end{remark}

\begin{remark}
\label{remark:inductive_definition_of_a_hypercolumn_cofibre_diagram}
If $r = 0$, then the condition of \cref{definition:hypercolumn_cofibre_diagram} is vacuous and any diagram $\sq^{\times 0} \rightarrow \ccat$ is a hypercolumn cofibre sequence. Such a diagram can be identified with a choice of an object of $\ccat$. 

If $r \geq 1$, then the decomposition $\sq^{\times r} \simeq \sq \times \sq^{\times r-1}$ associates to a diagram $P \colon \sq^{\times r} \rightarrow \ccat$ a curried diagram 
\[
\widetilde{P} \colon \sq \rightarrow \Fun(\sq^{\times r-1}, \ccat). 
\]
Then $P$ is hypercolumn cofibre if and only if $\widetilde{P}$
\begin{enumerate}
    \item is a cofibre sequence and 
    \item takes values in the subcategory $\Fun^{\mathrm{hpcofib}}(\sq^{\times r-1}, \ccat) \subseteq \Fun(\sq^{\times r-1}, \ccat)$ of hypercolumn cofibres.
\end{enumerate}
This gives an alternative, inductive definition of hypercolumn cofibres.
\end{remark}

\begin{construction}
\label{construction:shifted_n_cube}
If $\ccat$ is a stable $\infty$-category, then both of the restriction functors 
\[
\begin{tikzcd}
	{\Fun(\Delta^1, \ccat) } & {} & {\Fun^{\mathrm{cofib}}(\sq, \ccat)} & {} & {\Fun(\Delta^1, \ccat) }
	\arrow["{\res_{\{-1, 0\}}}", from=1-3, to=1-1]
	\arrow["{\res_{\{0, 1\}}}"', from=1-3, to=1-5]
\end{tikzcd}
\]
to the $1$-simplex spanned by either $\{ -1, 0 \}$ or $\{ 0, 1 \}$, are equivalences of $\infty$-categories. Their inverses are given by completing a morphism to, respectively, either a cofibre or fibre sequence. An inductive application of this statement and \cref{remark:inductive_definition_of_a_hypercolumn_cofibre_diagram} show that, more generally, for any $r \geq 0$ the restriction functors 
\[
\begin{tikzcd}
	{\Fun((\Delta^1)^{\times r}, \ccat) } & {} & {\Fun^{\mathrm{hpcofib}}(\sq^{\times r}, \ccat)} & {} & {\Fun((\Delta^1)^{\times r}, \ccat) }
	\arrow["{\res_{\{-1, 0\}^{\times r}}}", from=1-3, to=1-1]
	\arrow["{\res_{\{0, 1\}^{\times r}}}"', from=1-3, to=1-5]
\end{tikzcd}
\]
are equivalences. We will refer to the composite autoequivalence 
\[
\mathrm{sh} \colonequals \res_{\{-1, 0\}^{\times r}} \circ (\res_{\{0, 1\}^{\times r}})^{-1} \colon \Fun((\Delta^1)^{\times r}, \ccat) \rightarrow \Fun((\Delta^1)^{\times r}, \ccat)
\]
as the hypercube \emph{shift}. 
\end{construction}

\begin{example}
Unwrapping the definiton, we see that the shifted hypercube of \cref{construction:shifted_n_cube} can be calculated by iterated fibres. For example, in the case of $r = 2$ it associates to a square 
\[
\begin{tikzcd}
	{X_{0, 0}} & {X_{0, 1}} \\
	{X_{1, 0}} & {X_{1, 1}}
	\arrow[from=1-1, to=1-2]
	\arrow[from=1-1, to=2-1]
	\arrow[from=1-2, to=2-2]
	\arrow[from=2-1, to=2-2]
\end{tikzcd}
\]
the square 
\[
\begin{tikzcd}
	{\mathrm{fib}(\mathrm{fib}(X_{0, 0} \rightarrow X_{0, 1}) \rightarrow \mathrm{fib}(X_{1, 0} \rightarrow X_{1, 1}))} & {\mathrm{fib}(X_{0, 0} \rightarrow X_{1, 0})} \\
	{\mathrm{fib}(X_{0, 0} \rightarrow X_{0, 1})} & {X_{0, 0}}
	\arrow[from=1-1, to=1-2]
	\arrow[from=1-1, to=2-1]
	\arrow[from=1-2, to=2-2]
	\arrow[from=2-1, to=2-2]
\end{tikzcd}
\]
\end{example}

\begin{remark}
\label{remark:total_cofibre_of_the_shift_functor}
By induction on $r$, one can verify that given a hypercube $P \colon (\Delta^{1})^{\times r} \rightarrow \ccat$, there is a canonical equivalence 
\[
\tcofib(\mathrm{sh}(P)) \simeq P(1, \ldots, 1) 
\] 
between the total cofibre of the shifted hypercube and the value of $p$ at the final vertex. 
\end{remark}

\begin{remark}
\label{remark:hypercube_shift_forward_a_desuspension_of_a_shift_backward}
Let $p \colon (\Delta^{1})^{\times r} \rightarrow \ccat$ be a hypercube diagram, where as is standard we identify $\Delta^{1}$ with the poset $\{ 0, 1 \}$. Then we have canonical equivalences 
\[
(\mathrm{sh}(p))(0, \ldots, 0) \simeq (\mathrm{sh}^{-1}(p))(1, \ldots, 1)[-r]
\]
of objects of $\ccat$. This follows by induction from the observation that in a stable $\infty$-category, the fibre of a morphism can be identified with the desuspension of the cofibre. 
\end{remark}

\begin{proposition}
\label{proposition:formula_for_decalaged_pole_order_filtration}
Let $(X,D)$ be a projective sncd pair over a scheme $S$. Then,
\begin{align}
    D_* \R\Gamma(X; \Omega^\bullet_{X/S}(\log D)) &\simeq \tcofib_{I \subset [r]} (\tau_{\geq -\ast}\R\Gamma(D_I; \Omega^\bullet_{D_I/S})^\vee[-2n]) \in \Fil_\up(\D(S)) \\
    D_* \R\Gamma(X; \Omega^{\geq *}_{X/S}(\log D)) &\simeq \tcofib_{I \subset [r]} (\tau_{\geq -\ast}\R\Gamma(D_I; \Omega^{\geq *}_{D_I/S})^\vee[-2n]\{n\}) \in \Fil_\up(\DF(S)) \\
    D_* \R\Gamma(X; \Omega^*_{X/S}(\log D)) &\simeq \tcofib_{I \subset [r]}(\tau_{\geq -\ast}\R\Gamma(D_I; \Omega^{*}_{D_I/S})^\vee[-2n](n)) \in \Fil_\up(\DGr(S)), 
\end{align}
naturally in divisor enlargements, where $\tau_{\geq \ast}$ denotes the Whitehead filtration, and the structure maps on the hypercube diagram indexed by $I$ are given by duals of pullbacks.
\end{proposition}

\begin{proof}
We prove the claim for logarithmic Hodge filtered de Rham cohomology, as the proof for the other two cohomology theories is essentially the same. Let $(X,D)$ be a projective sncd pair over $k$, and let $D_1 + \cdots + D_r$ be the decomposition of $D$ into its components. We will identify 
\[
\mathrm{P}(\{ 1, \ldots, r \}) \simeq (\Delta^{1})^{\times r} 
\]
the poset of subsets of $\{ 1, \ldots, r \}$ with the $r$-dimensional hypercube. We then have a hypercube diagram 
\[
A \colon (\Delta^{1})^{\times r} \rightarrow \Fil^{\down}(\DF(S))
\]
given by 
\[
I \mapsto D_* \R\Gamma (X; \Omega^{\geq *}_{X/S}(\log D(I)))  
\]
where $D(I) = \sum_{i \in I} D_i$ and where the transition maps are given by pullbacks $j^*$ of \cref{recollection:residue_and_localization_sequences}. 

Let $G \colonequals \mathrm{sh}(A)$ be the Gysin hypercube obtained from $A$ by applying the shift functor of \cref{construction:shifted_n_cube}. The residue exact sequence diagram of \cref{example:hypercolumn_cofibre_diagram_from_residue_sequences} provides a hypercube cofibre extension of $A$ and thus its appropriate restriction can be identified with the inverse shift $\mathrm{sh}^{-1}(A)$. Using this identification and \cref{remark:hypercube_shift_forward_a_desuspension_of_a_shift_backward} one sees that on objects, $G$ is given by 
\begin{equation}
\R\Gamma(D_{J}; \Omega^{\geq * - |J|}_{D_{J}/S})[-2|J|] = \R\Gamma(D_{J}; \Omega^{\geq *}_{D_{J}/S})[-2|J|]\{|J|\},
\end{equation}
where we write $J \colonequals \{ 1, \ldots, r \} \setminus I$ for the complement. Moreover, by \cref{lem:DecTurnsInsToWhitehead} the resulting filtration on these values are given by degreewise Whitehead filtrations. As a consequence of \cref{remark:total_cofibre_of_the_shift_functor}, we have 
\[
G(1, \ldots, 1) \simeq D_* \R\Gamma (X; \Omega^{\geq *}_{X/S}(\log D)) \simeq \tcofib(G).
\]
Thus, to prove the needed result, it suffices to identify $G$ with the hypercube in the statement of the proposition.

Since the Whitehead filtration functor is fully faithful as we observe in \cref{theorem:whitehead_miracle}, it suffices to identify the two hypercubes after forgetting the $D_*$-filtration. By Poincaré duality of \cref{thm:PoincarePlus}, the dual $A^{\vee}$ is the hypercube diagram given by 
\[
I \mapsto \Omega^{\geq *}_{X/S,c}(\log D(I))[2n]\{-n\},
\]
with transition map given by pushforwards $j_*$ of \cref{recollection:residue_and_localization_sequences}. Moreover, the hypercube of localization sequences of \cref{construction:residue_and_localization_sequence_diagrams} identifies $\mathrm{sh}^{-1}(A^{\vee})$ with the hypercube diagram 
\begin{equation}
I \mapsto \Omega^{\geq *}_{D_I/S}[2n]\{-n\}.
\end{equation}
with transition maps given by pullback. Since
\[
(\mathrm{sh}^{-1})(A^{\vee}))^{\vee} \simeq \mathrm{sh}(A^{\vee})^{\vee}) \simeq \mathrm{sh}(A) \simeq G, 
\]
this identifies $G$ with the hypercube as in the statement. As all steps in the argument were natural with respect to divisor enlargements, the naturality of the equivalences follows. 
\end{proof}

\begin{lemma}
\label{lemma:mapping_into_hodge_filtered_de_rham_strongly_monoidal}
Let $\mathrm{H}^* \dR(S) \in \CAlg(\Fil^{\down}(\DM^\kgl_S))$ be the Hodge filtered de Rham cohomology motivic spectrum of \cref{definition:hodge_filtered_de_rham_cohomology_and_its_cousins}. Then the lax symmetric monoidal structure on the mapping spectrum functor 
\[
\mathrm{H}^* \dR(S)(-) \colon (\DM^\kgl_{S, \res})^{op} \to \Fil^{\down}\DF(S).
\]
determined by the algebra structure of $\mathrm{H}^* \dR(S)$ is (strongly) symmetric monoidal. In particular, it takes duals to duals. 
\end{lemma}

\begin{proof}
In the case of motives $\Mbf_{\kgl}(X), \Mbf_{\kgl}(X')$ of smooth projective $S$-schemes $X, X'$, the canonical comparison map
\[
\mathrm{H}^* \dR(S)(-) \otimes_{\mathcal{O}_{S}} \mathrm{H}^* \dR(S)(-) \rightarrow \mathrm{H}^* \dR(S)(- \otimes -) 
\]
can be identified with 
\[
\R\Gamma(X, \Omega^{\geq \ast}_{X/S}) \otimes_{\mathcal{O}_{S}} \R\Gamma(X', \Omega^{\geq \ast}_{X'/S}), \rightarrow \R\Gamma(X \times_{S} X', \Omega^{\geq \ast}_{X \times_{S} X'/S}), 
\]
which is an equivalence. Since Tate twists are taken to an appropriate twist in filtered spectra, namely 
\begin{equation}
\label{equation:hodge_filtered_de_rham_of_tate_twists}
\mathrm{H}^*\dR(S)\big((-)(-a)\big) \simeq  \mathrm{H}^*\dR(S)(-)[2a]\{-a\}
\end{equation}
for every $a \in \mathbb{Z}$, this comparison map is also an equivalence when both inputs are perfect pure. Since $\mathrm{H}^{*} \dR(S)(-)$ is exact in each variable, and perfect pure motives generate $\DM^{\kgl}_{S, \res}$ as a thick subcategory, we deduce that the functor is strongly monoidal. 
\end{proof} 

\begin{proof}[Proof of \cref{thm:DeligneCompare}]
We prove the claim for logarithmic Hodge filtered de Rham cohomology, as the proof for the other two cohomology theories is essentially the same. Since $W_{*} \mathrm{H}^{*} \dR(S)(-)$ is exact, by the definition of the logarithmic motive we have 
\[
W_* \mathrm{H}^*\dR(S)(\Mbf_\kgl(X,D)) \simeq \tcofib_{I \subset [r]}  W_* \mathrm{H}^*\dR(S)(\Mbf_\kgl(D_I)(-n))^\vee\big) \in \Fil_\up(\DF(S)).
\]
Using that $ \Mbf_\kgl(D_I)(-n))^\vee$ are perfect pure since the latter are closed under taking duals and Tate twists, by the defining property of $W_*$ this is equivalent to 
\[
\tcofib_{I \subset [r]}  \tau_{\geq - \ast} \big( \mathrm{H}^*\dR(S)(\Mbf_\kgl(D_I)(-n))^\vee\big). 
\]
Using (\ref{equation:hodge_filtered_de_rham_of_tate_twists}) to evaluate $\mathrm{H}^{*} \dR(S)(-)$ on the Tate twists and \cref{lemma:mapping_into_hodge_filtered_de_rham_strongly_monoidal} to commute it with taking duals, we can further rewrite this as 
\[
W_* \mathrm{H}^*\dR(S)(\Mbf_\kgl(X,D)) \simeq \tcofib_{I \subset [r]} \tau_{\geq - \ast}\big(\R\Gamma(D_I; \Omega^{\geq *}_{D_I/S})^\vee[-2n]\{n\}\big).
\]
This equivalence is clearly natural in divisor enlargemnets. The left hand side can be naturally identified with $D_* \R\Gamma(X; \Omega^{\geq *}_{X/S}(\log D)) \in \Fil_\up(\DF(S))$ by \cref{proposition:formula_for_decalaged_pole_order_filtration}, ending the argument. 
\end{proof}
\section{Further Topics}
\label{sect:misc}

In this last section, we give examples, describe a variant of the weight filtration in the setting of singular schemes, and list a few open problems which arose during the writing of the current work. 

\subsection{Examples}

We compute several examples of invariants that can be obtained from the weight filtration. Before doing so, we make the following definition.

\begin{definition}
Let $(X,D)$ be a strict normal crossing divisor pair of dimension $n$ over $S$. The \emph{compactly supported logarithmic motive} is the Tate twisted dual $\Mbf_{\kgl}(X,D)^\vee(n)$ of the logarithmic motive, i.e., the total cofibre 
\[
    \Mbf_{\kgl,c}(X,D) \colonequals \tcofib_{I \subset [r]} (\Mbf_{\kgl}(D_I)) \in \DM^\kgl_{S,\res},
\]
of the hypercube of the canonical inclusions between $D_I = \bigcap_{i \in I} D_i$.
\end{definition}

\begin{example}\label{ex:bdry}
Let $S$ be an affine Dedekind scheme, let $R$ be a ring, and let $\mathrm{H}R \in \SH_S$ be the motivic spectrum representing the $R$-coefficient motivic cohomology. In this case, for a projective and connected sncd pair $(X,D)$ over $S$, we have that
\begin{align}
    \Gr^W_0 \mathrm{H}R(\Mbf_{\kgl,c}(X,D)) &= \tfib_{I \subset [r]} \Gamma(D_I;\underline{R}) \\
    &\simeq \widetilde{\mathrm{H}}^\bullet(\Delta(D); R)
\end{align}
is the $R$-coefficient reduced singular cohomology of the dual complex $\Delta(D)$ associated to the snc divisor $D$ (see e.g. \cite[\S 2]{payne:2013} for the definition of $\Delta(D)$).
\end{example}

As $\Mbf_{\kgl,c}(X,D)$ is an invariant of $U = X-D$, we have proven the following result, which over a perfect field follows also from Thuillier's result stating that the homotopy type of $\Delta(D)$ is an invariant of $U$ \cite{thuillier:2007}.

\begin{theorem}
Let $U$ be a smooth scheme over an affine Dedekind scheme $S$. If $U$ admits a good projective compactification over $S$, then the singular cohomology complex of its dual complex is an invariant of $U$.
\end{theorem}

In particular, if $X \in \Sch$ is projective, and admits a strong resolution of singularities $(X',E)$ over $S$, then the singular cohomology of $\Delta(E)$ is an invariant of $X$ because $X_\sm \cong X'-E$.

The next example illustrates that the weight-zero pieces contain more than just combinatorial information about the boundary complex.

\begin{example}
Let $(X,D)$ be and sncd pair over $S$. Then,
\begin{equation}
    \Gr^W_{j+i} \Hod^j(S)(\Mbf_{\kgl,c}(X,D)) = \tfib_{I \subset [r]} \R\Gamma^i(D_I;\Omega_{D_I/S}^j)[-j-i] \in \D(S).
\end{equation} 
In particular, the right hand side is an invariant of $U = X-D$.

For example, consider the divisors $D_1, D_2 \subset \Pbf^2_{\Fb_5}$ cut out by $x_0^2 - 2x_2^2$ and $x_1-x_2$, and let $D = D_1+D_2$. The only global functions on $\Pb^2_{\Fb_5},$ $D_1$, and $D_2$ are the elements of $\Fb_5$, but the global functions on $D_1 \cap D_2$ form a degree-two field extension over $\Fb_5$. Thus,
\begin{equation}\label{eq:ExplicitExampleWtOfHodge}
    \Gr^W_0 \Hod^0(\Fb_5)(\Mbf_{\kgl,c}(X,D)) \simeq \Fb_5[-2] \in \D(\Fb_5).
\end{equation}
From this we can conclude that $U = \Pbf^2_{\Fb_5} - D$ does not admit a good projective compactification with only one boundary component, as such a compactification could not result in a $\Gr^W_0 \Hod^0(\Fb_5)$ as above.
\end{example}

\subsection{Weight filtrations on singular schemes}

We have described a method endowing various at-the-characteristic cohomology theories of smooth schemes with weight filtrations, thereby producing new invariants of such schemes. Here, we show that our methods also apply to singular schemes.

Let $S$ be a base scheme. Due to the work of Cisinski~\cite{cisinski:2019}, it is known that there is a fully faithful, symmetric monoidal  embedding
\begin{equation}
    (-)^\cdh \colon \SH_S \hook \SH^\cdh_S
\end{equation}
where $\SH^\cdh_S$ is the variant of stable $\Abf^1$-homotopy category obtained by using cdh sheaves on the category $\Sch_S$ of finitely presented $S$-schemes (see e.g. \cite{khan:cdh}). In particular, an $\Abf^1$-invariant cohomology theory $E$ on $\Sm_S$ has an extension into an $\Abf^1$-invariant, $\cdh$-local cohomology theory $E^\cdh$ on $\Sch_S$. By symmetric monoidality and \cref{prop:PureToRes}, $\DM^\kgl_{S,\res} \subset \Mod_{\kgl^\cdh}(\SH^\cdh_S)$. If $X \in \Sch_S$, we denote by $\Mbf_\kgl^\cdh(X)$ the $\cdh$-local $\Abf^1$-invariant $\kgl^\cdh$-motive of $X$.

Suppose then that $X \in \Sch_S$ is such that $\Mbf_{\kgl}^\cdh(X)$ is resolvable. This is the case for example if the reduction $X_\red$ admits a proper birational map $\pi \colon X' \to X_\red$ such that $X'$ is a smooth $S$-scheme that admits a good projective compactification over $S$, and $\pi^{-1} X_\sing$ is an snc divisor. We call such a resolution a \emph{strong compactifiable resolution of singularities over $S$.} Then \cref{thm:IndependenceOfWts} applies to it so we can endow various cohomologies of $X$ with a weight filtration. We illustrate this general principle with the following example.

\begin{example}[``du Bois'' theory in positive and mixed characteristic]\label{ex:dB}
Let $S$ be an affine Dedekind scheme and consider the functors out of $\DM^\kgl_{S,\res}$ given by $W_*\dR(S), W_*\mathrm{H}^*\dR(S),$ and $W_*\Hod^*(S)$. They define cohomology theories on resolvable motives satisfying that
\begin{enumerate}
    \item if $X$ is smooth and projective over $S$, then the cohomology evaluated on $\Mbf^\cdh_\kgl(X)$ coincides with de Rham cohomology (resp. Hodge filtered de Rham cohomology, resp. Hodge cohomology) of $X$ over $S$ with the negatively indexed Whitehead filtration;
    \item if $U$ is smooth over $S$ and admits a projective sncd compactification $(X,D)$ over $S$, then the cohomology evaluated on $\Mbf^\cdh_\kgl(U)$ coincides with the logarithmic de Rham cohomology (resp. logarithmic Hodge filtered de Rham cohomology, resp. logarithmic Hodge cohomology) of $(X,D)$ over $S$ with the décalaged pole-order filtration (\cref{thm:DeligneCompare});
    \item if $Y \in \Sch_S$ is singular but admits a strong compactifiable resolution of singularities over $S$, then the cohomology evaluated on $\Mbf^\cdh_\kgl(Y)$ may be computed using this resolution and $\cdh$-descent. The cohomology and its weight filtration do not depend on the choice of such a resolution.
\end{enumerate}
The cohomology theory represented by $W_*\Hod^*(S)$ on singular schemes with resolvable motives is reminiscent of the cohomology of the du Bois complex from characteristic 0 \cite{dubois:1981}. Park and Popa have recently used the dimensions of the cohomology groups of the du Bois complex (\textit{Hodge-du Bois numbers}) in order to study singularities in characteristic 0 \cite{park-popa:hodge_symmetry, park-popa:qfactoriality}. The authors wonder if $W_*\dR(S), W_*\mathrm{H}^*\dR(S),$ and $W_*\Hod^*(S)$ could be similarly used to study singularities in positive and mixed characteristic. In \cref{ex:singular_example}, we compute the weight-filtered $W_*\Hod^*(S)$ for projective cones over smooth varieties.
\end{example}

\begin{remark}[$\Abf^1$-colocalization]\label{rem:A1colocalization}
We have chosen to state the results of this article without using the ($\kgl$-linear) \textit{$\Abf^1$-colocalization} functor from \cite[\S 6]{AHI:atiyah}, i.e., the right adjoint $(-)^\dagger \colon \DM^\kgl_S \to \DM^{\kgl, \Abf^1}_S$ of the inclusion $\DM^{\kgl, \Abf^1}_S \hook \DM^\kgl_S$. Another way to state \cref{thm:DeligneCompare} and \cref{ex:dB} would be that $\dR(S)^\dagger$ and $\dR(S)^{\dagger,\cdh}$ represent an $\Abf^1$-invariant Nisnevich-local cohomology theory on $\Sm_S$, and an $\Abf^1$-invariant cdh-local cohomology theory on $\Sch_S$, respectively, such that on resolvable motives it admits a weight filtration that  can be related to logarithmic de Rham cohomology with décalaged pole-order filtration (and similarly for Hodge filtered de Rham and Hodge cohomology).
\end{remark}

\begin{example}[Cones over smooth projective schemes]\label{ex:singular_example}
Let $X$ be smooth and projective over a Dedekind domain $A$, and suppose for simplicity that the only global functions are the elements of $A$. Let $\overline{C}(X,\Lc)$ be the projective cone associated to an ample line bundle, and let 
\begin{equation}
\begin{tikzcd}\label{eq:ConeRes}
    X \arrow[r,hook]  \arrow[d] & \Pbf_X(\Lc \oplus \Oc) \arrow[d] \\
    \Spec(A) \arrow[r,hook] & \overline{C}(X,\Lc)
\end{tikzcd}
\end{equation}
be the resolution of singularities over $A$ obtained by blowing up the cone point. As \cref{eq:ConeRes} is an abstract blowup square, we may use it to compute that
\begin{equation}
    \Gr^W_{j+i} \Hod^j(A)(\Mbf^\cdh_{\kgl}(\overline{C}(X,\Lc))) = 
\begin{cases}
A& \text{if $i = j = 0$;}\\
\R\Gamma^{i-1}(X; \Omega^{j-1}_{X/A})][-j-i] & \text{otherwise.}
\end{cases}
\end{equation}
\end{example}

\subsection{Open problems}

In this section, we list several question which arose naturally during the writing of the current work, and which we believe would be interesting to pursue on their own right. 

\begin{quest}[{Crystalline comparison}]
\emph{Does the weight filtration introduced in the current work coincide with the décalaged pole-order filtration on crystalline cohomology of \cite{mokrane:1993,nakkajima2008weight}?}
\end{quest}

\begin{quest}[{Derived generalization}]
\emph{Does there exist a generalization of notions of compactly-supported differential forms to derived schemes so that Poincaré duality of \cref{thm:PoincarePlus} holds? What about the pole-order filtration and \cref{proposition:formula_for_decalaged_pole_order_filtration}?} 
\end{quest}

\begin{quest}[{Resolvable motives in terms of log geometry}]\label{quest:ResMots}
Our original strategy for proving \cref{thm:DeligneCompare} was to define the décalaged pole-order filtration $D_*$ on variants of de Rham cohomology as functors out of either
\begin{enumerate}
\item the category of logarithmic motives $\log\DM_S$ or 
\item the logarithmic motivic homotopy category $\log\SH_S$ \cite{binda:2022,binda:2023}. 
\end{enumerate}
The comparison would then follow from the universal property of weight filtrations of \cref{thm:IndependenceOfWts} by directly comparing them as functors out of resolvable motives. The key observation here is that the symmetric monoidal functor $\MS_S \to \log\SH_S$ can be used to embed resolvable motives into logarithmic motives, and that motives of log schemes associated to projective sncd pairs are dualizable due to the motivic residue sequences of \cite[Theorem~3.2.19]{binda:2023}.

Unfortunately, the strategy does not work because after equipping variants of de Rham cohomology with the $D_*$-filtration they fail to satisfy Nisnevich (or even Zariski) descent. In fact, descent fails already for the standard open cover of $\Pbf^1$ by two affine lines. This failure inspired as to ask the following question: 

\begin{center}
\emph{Is there a simple description of $\DM^\kgl_{S,\res}$  in terms of only smooth log smooth projective $S$-schemes?} 
\end{center}

Note that such a description cannot mention Nisnevich or Zariski topology, as these topologies cannot be described using only projective schemes. Our conjecture is that such a description exists and that it looks something like this:
\begin{enumerate}
    \item start from the category $\Pc$ of spectrum-valued presheaves on smooth log smooth projective $S$-schemes satisfying some conditions (e.g. smooth blowup excision);

    \item define $\mathcal{M}$ by formally inverting an explicit collection of objects in $\Pc$ (e.g. all Thom spaces over $S$);

    \item consider $\Dc := \Mod_{\kgl}(\mathcal{C})$. In particular, $\kgl$ should be representable by a commutative algebra in $\mathcal{M}$.
\end{enumerate}
Then, if Atiyah duality \cite[Corollary~5.15]{AHI:atiyah} holds in $\Dc$, $\DM^\kgl_{S,\res}$ is the smallest stable subcategory of $\Dc$ that contains all the Tate twists of smooth projective varieties (\cref{prop:PureToRes}). Moreover, if the category admits motivic residue sequences that allow expressing the motives of logarithmic schemes in terms of ordinary schemes, then the $\Dc$-motives of smooth log smooth projective $S$-schemes would be contained in $\DM^\kgl_{S,\res}$.

A simple description as described above would be useful in defining functors out of $\DM^\kgl_{S,\res}$ using logarithmic geometry, thus providing greater convenience in constructing invariants of resolvable motives. This idea seems reminiscent of Bittner's presentation of the Grothendieck ring of varieties over fields of characteristic zero whose generators are smooth proper varieties and whose generators come from blowup squares, see \cite{bittner:2004}.
\end{quest}

\begin{quest}[{du Bois theory in mixed characteristic}]
The cohomology theory represented by $W_*\Hod^*(S)$ on singular schemes as in \cref{ex:dB} is reminiscent of the cohomology of the du Bois complex \cite{dubois:1981}. Park and Popa have recently used the dimensions of the cohomology groups of the du Bois complex, the \textit{Hodge-du Bois numbers}, in order to study singularities in characteristic zero \cite{park-popa:hodge_symmetry, park-popa:qfactoriality}. This naturally leads to the following question: 

\begin{center}
\emph{Can the cohomology theories $W_*\dR(S), W_*\mathrm{H}^*\dR(S)$ and $W_*\Hod^*(S)$ be used to study singularities in positive and mixed characteristic?} 
\end{center}

\end{quest}
\appendix

\section{The Whitehead tower functor is fully faithful} 

The following result is folklore, but we could not find a suitable reference: 

\begin{theorem}
\label{theorem:whitehead_miracle}
The Whitehead tower functor
\[
\tau_{\geq \ast} (-) \colon \spectra \rightarrow \Fil^\down (\spectra)
\]
is fully faithful and its essential image is the full subcategory of those filtered spectra $F^{\ast} X$ such that for each $q \in \mathbb{Z}$ the canonical map 
\[
F^{q}X \rightarrow \varinjlim F^{\ast} X
\]
is a $q$-connective cover of spectra 
\end{theorem}

Our arguments are analogous to \cite[{Corollary 4.38}]{pstrkagowski2023synthetic}, where instead of the standard t-structure on synthetic spectra we use the following t-structure on the $\infty$-category $\Fil^\down(\spectra) \colonequals \Fun(\mathbb{Z}^{op}, \spectra)$ of filtered spectra: 

\begin{proposition}
\label{proposition:the_existence_of_the_diagonal_t_structure}
Let us say that a filtered spectrum $F^{\ast} X$ is
\begin{enumerate}
    \item \emph{diagonal connective} if $F^{q} X$ is $q$-connective for each $q \in \mathbb{Z}$,
    \item \emph{diagonal coconnective} if $F^{q} X$ is $q$-coconnective for each $q \in \mathbb{Z}$. 
\end{enumerate}
Then the above pair of subcategories defines a t-structure on the $\infty$-category of filtered spectra. 
\end{proposition}

\begin{proof}
Consider the bifiltered spectrum
\[
\tau_{\geq \ast} F^{\ast} Z \colon \mathbb{Z}^{op} \times \mathbb{Z}^{op} \rightarrow \spectra
\]
given by applying the Whitehead tower functor level-wise. Composing with the diagonal map of $\mathbb{Z}^{op}$ we obtain a filtered spectrum $\mathrm{diag}^{*}(\tau_{\geq \ast} F^{\ast} Z)$ which in filtered degree $q$ is given by $\tau_{\geq q} F_{q} Z$. Thus, this filtered spectrum is diagonal $0$-connective and the cofibre of the canonical map 
\[
\mathrm{diag}^{*}(\tau_{\geq \ast} F^{\ast} Z) \rightarrow F^{\ast} Z 
\]
is diagonal $(-1)$-coconnective. 

Thus, to show that the two notions define a t-structure, we are left with verifying that if $F^{\ast} X$ is diagonal $0$-connective and $F^{\ast} Y$ is diagonal $0$-coconnective, then the mapping space $\Map(F^{\ast} X, F^{\ast} Y)$ is discrete. However, by \cite[{Proposition 5.1}]{gepner2017lax} this mapping space can be calculated using the end formula; that is, we have 
\[
\Map(F^{\ast} X, F^{\ast} Y) \simeq \varprojlim \Map_{\spectra}(F^{q_{1}} X, F^{q_{2}} Y)
\]
where the limit is taken over $\mathrm{Tw}(\mathbb{Z}^{op})$, the twisted arrow $\infty$-category. Since the objects of the twisted arrow $\infty$-category are arrows in $\mathbb{Z}^{op}$, we can identify them with pairs $q_{1}, q_{2}$ of integers such that $q_{1} \geq q_{2}$. In this case, $F^{q_{1}} X$ is $q_{1}$-connective and $F^{q_{2}} Y$ is $q_{2}$-coconnective by assumption, so under the assumption that $q_{1} \geq q_{2}$, the relevant mapping space is discrete. Thus, the limit is discrete as well, as needed. 
\end{proof}

\begin{proof}[{Proof of \cref{theorem:whitehead_miracle}}]
We have to show that if $X, Y$ are spectra, then the canonical morphism  
\[
\Map_{\spectra}(X, Y) \rightarrow \Map_{\Fil^\down(\spectra)}(\tau_{\geq \ast} X, \tau_{\geq \ast} Y) 
\]
between mapping spaces is an equivalence. Since the composite 
\[
\begin{tikzcd}
	{\spectra} & {\Fil^\down(\spectra)} & {\spectra}
	\arrow["{\tau_{\geq \ast}}", from=1-1, to=1-2]
	\arrow["{\varinjlim }", from=1-2, to=1-3]
\end{tikzcd}
\]
can be identified with the identity of $\spectra$, we can equivalently show that 
\[
\Map_{\Fil^\down(\spectra)}(F^{\ast} X, F^{\ast} Y) \rightarrow \Map_{\spectra}(\varinjlim F^{\ast} X, \varinjlim F^{\ast} Y).
\]
is an equivalence. Since taking colimits is a left adjoint to the constant filtration functor, the claim is equivalent to showing that the morphism
\begin{equation}
\label{equation:equivalence_induced_by_unit_for_whitehead_towers}
\Map_{\Fil^\down(\spectra)}(\tau_{\geq \ast} X, \tau_{\geq \ast} Y) \rightarrow \Map_{\Fil^\down(\spectra)}(\tau_{\geq \ast} X, \mathrm{const}^{\ast}(Y))
\end{equation}
induced by the unit $\tau_{\geq \ast} Y \rightarrow \mathrm{const}^{\ast}(Y)$, where we implicitly identify $Y \simeq \varinjlim (\tau_{\geq \ast} Y)$, is an equivalence. This follows from the fact that $\tau_{\geq \ast} X$ is connective and the unit is a connective cover in the t-structure of \cref{proposition:the_existence_of_the_diagonal_t_structure}. This ends the argument that the Whitehead tower functor is fully faithful.

To identify the essential image, suppose that $F^{\ast} X$ is a filtered spectrum such that for each $q \in \mathbb{Z}$, the map $F^{q} X \rightarrow \varinjlim F^{\ast} X$ is a $q$-connective cover of spectra. Then both $F^{\ast} X$ and $\tau_{\geq \ast} (\varinjlim F^{\ast} X)$ can be identified with connective covers of $\mathrm{const}_{\ast}(\varinjlim F^{\ast} X)$ with respect to the t-structure of \cref{proposition:the_existence_of_the_diagonal_t_structure} and so they are equivalent. Thus, $F^{\ast} X$ is in the essential image. 
\end{proof}

\section{Bivariant pairings and dual functors} 

The following result is useful for recognizing functors of opposite variance as duals of one another by using pairings that are functorial with respect to the twisted arrow category. Although the statement is certainly known to experts, we are not aware of any citable reference. 

\begin{theorem}
\label{prop:functorial_duals}
Let $K$ be an $\infty$-category, $\Cc$ a symmetric monoidal $\infty$-category, let $F \colon K^\op \to \Cc$ and $G \colon K \to \Cc$ be functors, and suppose that $G$ takes values in dualizable objects. Let $A$ be the composition 
\[
\Tw(K) \to K^\op \times K \xto{F \otimes G} \Cc,
\]
and let $\mathrm{const}_{\mathbf{1}} \colon \Tw(K) \to \Cc$ be the constant, unit-valued functor to $\Cc$. Then, a natural transformation of the form 
\[
P \colon A \to \mathrm{const}_{\mathbf{1}} 
\]
determines and is determined by a natural transformation
\[
\widetilde{P} \colon F \to G^\vee.
\]
Furthermore, $P$ is a perfect pairing on identity morphisms if and only if the associated $\widetilde{P}$ is an equivalence. 
\end{theorem}

\begin{proof}
We claim that to give a natural transformation into the constant diagram is the same as to give a point in the limit of the functor $\Tw(K) \rightarrow \spaces$ given by 
\begin{equation}
\label{equation:space_of_perfect_pairings}
(a \to b) \mapsto \Map_{\Cc}(F(a) \otimes G(b), \mathbf{1}), 
\end{equation}
where $a \to b$ is a morphism in $K^{\op}$. To see this, note that by replacing $\ccat$ by its $\infty$-category of presheaves, we can assume that it has colimits. Then, since mapping into the constant diagram is right adjoint to taking the colimit, we have that 
\[
\mathrm{Nat}(P, \mathrm{const}_{\mathbf{1}}) \simeq \Map_{\ccat}(\varinjlim F(a) \otimes G(b), \mathbf{1}) \simeq \varprojlim \Map_{\ccat}(F(a) \otimes G(b), \mathbf{1})
\]
as needed. Since $G$ is assumed to take values in dualizable objects, we can rewrite the functor of (\ref{equation:space_of_perfect_pairings}) as 
\[
(a \to b) \mapsto \Map_{\Cc}(F(a) , G(b)^\vee).
\]
By \cite[Proposition~5.1]{gepner2017lax}, the limit of this diagram can be identified with the space of natural transformations $F \to G^\vee$, with the values at the identities giving the components of the natural transformation. This gives the last part, since to say that $P$ is a perfect pairing on the identity $\mathrm{id}_{a}$ corresponds to the condition that $F(a) \to G(a)^\vee$ is an equivalence. 
\end{proof}

\bibliographystyle{alphamod}
\bibliography{references}

\end{document}

%% file: preamble.tex
\usepackage[T1]{fontenc}
\usepackage[utf8]{inputenc}
\usepackage{amsmath, amsthm, amsfonts, amssymb}

\usepackage{euscript} 
\usepackage[usenames, dvipsnames]{xcolor} 
\usepackage{url}
\usepackage{nicefrac}
\usepackage[all]{xy}
\usepackage{placeins} 
\usepackage{musicography} 

\usepackage{bbm} 
\usepackage{enumitem} 

\usepackage{hyperref}
\hypersetup{
    colorlinks, linkcolor=NavyBlue,
    citecolor=OliveGreen, urlcolor=RedOrange
}
\usepackage[nameinlink,capitalise,noabbrev]{cleveref} 
\usepackage{todonotes}
\setuptodonotes{fancyline, color=green!30}
\usepackage{faktor}

\usepackage{tikz-cd}
\usepackage{tikz}
\usetikzlibrary{arrows,backgrounds,
    decorations.pathreplacing,
    decorations.pathmorphing
}
\usetikzlibrary{matrix,arrows}
\usepackage{colonequals}
\usetikzlibrary{decorations.pathmorphing,fit}

\newcommand{\euscr}[1]{\EuScript{#1}} 
\newcommand{\acat}{\euscr{A}} 
\newcommand{\ccat}{\euscr{C}} 
\newcommand{\map}{\textnormal{map}} 
\newcommand{\spaces}{\euscr{S}} 

\newcommand{\ZZ}{\mathbf{Z}}
\newcommand{\HKR}{\mathrm{HKR}}

\newcommand{\spectra}{\euscr{S}p} 

\newcommand{\nocontentsline}[3]{}
\newcommand{\tocless}[2]{\bgroup\let\addcontentsline=\nocontentsline#1{#2}\egroup}

\theoremstyle{plain}

\newtheorem{theorem}{Theorem}[section]
\newtheorem{lemma}[theorem]{Lemma}
\newtheorem{proposition}[theorem]{Proposition}
\newtheorem{corollary}[theorem]{Corollary}

\theoremstyle{definition}
\newtheorem{example}[theorem]{Example}

\newtheorem{recollection}[theorem]{Recollection}

\newtheorem{definition}[theorem]{Definition}
\newtheorem{remark}[theorem]{Remark}
\newtheorem{notation}[theorem]{Notation}
\newtheorem{construction}[theorem]{Construction}

\newtheorem*{remark*}{Remark}
\newtheorem*{terminology*}{Terminology}
\newtheorem*{interpretation*}{Interpretation}
\newtheorem*{definition*}{Definition}
\newtheorem*{conjecture*}{Conjecture}
\newtheorem*{notation*}{Notation}
\newtheorem*{convention*}{Convention}

\theoremstyle{remark}

\numberwithin{equation}{section}

\makeatletter
  \def\subsection{\@startsection{subsection}{1}%
  \z@{.7\linespacing\@plus\linespacing}{.5\linespacing}%
  {\normalfont\bfseries\centering}}
\makeatother

\let\oldtocsection=\tocsection
\let\oldtocsubsection=\tocsubsection
\let\oldtocsubsubsection=\tocsubsubsection
\renewcommand{\tocsection}[2]{\hspace{0em}\oldtocsection{#1}{#2}}
\renewcommand{\tocsubsection}[2]{\hspace{1em}\oldtocsubsection{#1}{#2}}
\renewcommand{\tocsubsubsection}[2]{\hspace{2em}\oldtocsubsubsection{#1}{#2}}

\calclayout

\newcommand{\forgetful}{F}

\DeclareMathOperator{\THH}{THH}

\DeclareMathOperator{\TC}{TC}

\DeclareMathOperator{\sq}{sq}

\mathchardef\mhyphen="2D

%% file: macros.tex
\newcommand{\Ab}{\mathbb{A}}

\newcommand{\Fb}{\mathbb{F}}

\newcommand{\Pb}{\mathbb{P}}

\newcommand{\Zb}{\mathbb{Z}}

\newcommand{\Ac}{\mathcal{A}}

\newcommand{\Cc}{\mathcal{C}}
\newcommand{\Dc}{\mathcal{D}}

\newcommand{\Fc}{\mathcal{F}}

\newcommand{\Lc}{\mathcal{L}}

\newcommand{\Oc}{\mathcal{O}}
\newcommand{\Pc}{\mathcal{P}}

\newcommand{\Abf}{\mathbf{A}}

\newcommand{\Mbf}{\mathbf{M}}

\newcommand{\Pbf}{\mathbf{P}}

\newcommand{\Tbf}{\mathbf{T}}

\newcommand{\Zbf}{\mathbf{Z}}

\newcommand{\tr}{\mathrm{Tr}}
\newcommand{\op}{\mathrm{op}}

\newcommand{\Hom}{\mathrm{Hom}}
\newcommand{\Spec}{\mathrm{Spec}}
\newcommand{\Gr}{\mathrm{Gr}}

\newcommand{\QCoh}{\mathrm{QCoh}}

\newcommand{\Fun}{\mathrm{Fun}}
\newcommand{\hook}{\hookrightarrow}

\newcommand{\red}{\mathrm{red}}

\newcommand{\xto}{\xrightarrow}

\newcommand{\Map}{\mathrm{Map}}

\newcommand{\MGL}{\mathrm{MGL}}
\newcommand{\tto}{\twoheadrightarrow}

\newcommand{\MS}{\mathrm{MS}}
\newcommand{\SH}{\mathrm{SH}}

\newcommand{\Mod}{\mathrm{Mod}}
\newcommand{\Sm}{\mathrm{Sm}}

\newcommand{\sm}{\mathrm{sm}}

\newcommand{\mot}{\mathrm{mot}}
\newcommand{\Pure}{\mathrm{Pure}}
\newcommand{\DM}{\mathrm{DM}}
\newcommand{\Sp}{\mathrm{Sp}}
\newcommand{\Nis}{\mathrm{Nis}}

\newcommand{\res}{\mathrm{res}}

\newcommand{\fin}{\mathrm{fin}}
\newcommand{\exa}{\mathrm{ex}}
\newcommand{\crys}{\mathrm{crys}}
\newcommand{\kgl}{\mathrm{kgl}}
\newcommand{\KGL}{\mathrm{KGL}}
\newcommand{\Fil}{\mathrm{Fil}}

\usepackage{relsize}
\usepackage[bbgreekl]{mathbbol}
\DeclareSymbolFontAlphabet{\mathbb}{AMSb} 
\DeclareSymbolFontAlphabet{\mathbbl}{bbold}
\newcommand{\prism}{{\mathbbl{\Delta}}}

\renewcommand{\res}{\mathrm{res}}
\newcommand{\Dec}{\mathrm{Dec}}

\newcommand{\heart}{\heartsuit}

\newcommand{\D}{\mathrm{D}}
\newcommand{\R}{\mathrm{R}}
\newcommand{\Hod}{\mathrm{Hdg}}
\newcommand{\DF}{\mathrm{DF}}
\newcommand{\tcofib}{\mathrm{tcofib}}
\newcommand{\tfib}{\mathrm{tfib}}
\newcommand{\SNCD}{\mathrm{SNCD}}
\newcommand{\Tw}{\mathrm{Tw}}
\newcommand{\CAlg}{\mathrm{CAlg}}

\newcommand{\up}{\uparrow}
\newcommand{\down}{\downarrow}
\newcommand{\DGr}{\mathrm{DGr}}
\newcommand{\Sch}{\mathrm{Sch}}
\newcommand{\cdh}{\mathrm{cdh}}
\newcommand{\sing}{\mathrm{sing}}
\renewcommand{\log}{\mathrm{log}}

\newcommand{\dR}{\mathrm{dR}}
